\documentclass[11pt]{article}

\usepackage[a4paper,margin=1in]{geometry}
\usepackage{amsmath,amssymb,amsfonts,amsthm,mathtools}
\usepackage{graphicx,booktabs,array,microtype}
\usepackage{xcolor}
\usepackage{subcaption}
\usepackage[colorlinks=true,linkcolor=blue!55!black,
            citecolor=blue!55!black,urlcolor=blue!55!black]{hyperref}

\newtheorem{theorem}{Theorem}[section]
\newtheorem{proposition}[theorem]{Proposition}

\newtheorem{corollary}[theorem]{Corollary}
\theoremstyle{remark}
\newtheorem{remark}[theorem]{Remark}

\newcommand{\R}{\mathbb{R}}

\newcommand{\Z}{\mathbb{Z}}
\newcommand{\Sone}{\mathbb{S}^{1}}
\newcommand{\ii}{\mathrm{i}}
\newcommand{\dd}{\mathrm{d}}
\newcommand{\rank}{\operatorname{rank}}
\newcommand{\diag}{\operatorname{diag}}
\newcommand{\Nc}{N_{\rm c}}
\newcommand{\Nh}{N_{\rm h}}
\newcommand{\Ns}{N_{\rm s}}
\newcommand{\Nang}{N_{\rm ang}}
\newcommand{\Var}{\operatorname{Var}}

\title{Fourier--Hankel Moment Recovery in Acoustic Scattering:
Multichannel Stabilization and Radial Interface Resolution}

\author{
Zhiliang Deng\thanks{Corresponding author. School of Mathematical Sciences, University of Electronic
Science and Technology of China. Email: dengzhl@uestc.edu.cn}
\and
Xiaofei Guan\thanks{School of Mathematical Sciences, Tongji University.
Email: guanxf@tongji.edu.cn}
\and
Xiaomei Yang\thanks{School of Mathematics,
Southwest Jiaotong University. Email: yangxiaomath@swjtu.edu.cn}
}

\date{}

\begin{document}
\maketitle

\begin{abstract}
We study direct recovery of visible phase centers and concentric radial
interfaces from full-aperture acoustic far-field data under the Born
approximation.  Two complementary moment structures are extracted from the
two-angle Fourier matrix.  At fixed positive total Fourier order, the
nonnegative-order channels share the same leading phase-center moment, and a
generalized least-squares combination yields an exact variance gain over a
single Fourier row.  The corresponding Hankel rank and shifted pencil recover
distinct phase centers.  Signed moments reveal off-center cavities but become
degenerate when material and cavity centers coincide.  Zero-total-order
coefficients retain complementary radial Bessel moments.  For a
piecewise-constant radial average, low-frequency extrapolation produces a
second finite exponential sequence whose nodes are the squared interface
radii.  We establish rank and perturbation results for both reductions.
Numerical experiments verify multichannel stabilization, concentric-cavity
resolution, and recovery of multiple radial interfaces.  A final full-wave
Helmholtz experiment, generated without the Born substitution, assesses the
Born-derived reconstruction under model mismatch and shows how nonlinear
scattering eventually appears as an additional Hankel tail.
\end{abstract}


\section{Introduction}
\label{sec:introduction}

Inverse acoustic scattering seeks geometric or material information from waves
measured away from an unknown target.  With multistatic far-field data one may
reconstruct an entire contrast or boundary, or instead extract a smaller set
of geometric descriptors directly from the measurements.  The latter
viewpoint is natural when the quantities of interest are a component count,
representative locations, or the presence of a cavity rather than a full
image.  General treatments of qualitative and iterative inverse-scattering
methods can be found in
\cite{CakoniColton2014, ColtonKress2019, Kirsch2011, KirschGrinberg2008}.

Direct sampling and probe methods provide one route to geometric information
without solving a full nonlinear reconstruction problem
\cite{Potthast2006}.  In acoustic scattering, direct sampling methods construct
near- or far-field indicators for locating unknown media or obstacles
\cite{Ito2012,Li2013}.  The extended sampling method follows a related direct
philosophy, using the measured far field from a single incident wave to obtain
location and support information, with extensions to limited-aperture settings
\cite{Liu2018}.  These approaches recover geometric information through
spatially indexed indicators and are attractive because they avoid repeated
solution of a nonlinear inverse problem.

Other direct methods exploit finite-dimensional structure in the measured
scattering data.  Small-volume asymptotics encode target locations and
polarization information
\cite{AmmariIakovlevaMoskow2003,AmmariKang2004,AmmariKang2007},
while MUSIC, DORT, and related subspace methods use low-rank response matrices
to localize small scatterers and, in suitable regimes, estimate their number
\cite{Cheney2001,Devaney2000,FazliNakhkash2012,QuinlanBarbotLarzabal2006}.
Far-field splitting provides another route to the number and approximate
locations of separated components
\cite{GriesmaierHankeSylvester2014}.  More closely related to the use of
angular Fourier information, Hanke \cite{Hanke2012} combines Fourier
coefficients of a single scattered field with a Prony--Pad\'e construction
and repeated virtual origins, while Griesmaier and Hanke
\cite{GriesmaierHanke2025} use low-order Fourier coefficients and a tailored
rational approximation followed by localization from several virtual origins.
The approach developed here belongs to this broader family of direct
data-reduction methods, but exploits a different structure in the full
two-angle far-field data.  Rather than constructing a spatial sampling
indicator or using selected low-order coefficients, we organize the
two-angle Fourier coefficients into finite moment sequences and study the
resulting structured Hankel matrices.

Under the Born approximation, the Fourier coefficients of the two-angle
multistatic far field admit a Bessel--Fourier representation that gives rise
to two complementary finite-moment reductions in a fixed coordinate system.
The first reduction produces a phase-center sequence.  A single
incident-angle-averaged row already yields, at leading order, a finite
exponential sequence whose Hankel rank and shifted pencil encode the distinct
visible phase centers.  This row construction, however, becomes poorly scaled
at high Fourier order.  For each fixed positive total order, the nonnegative
Fourier pairs on the corresponding anti-diagonal contain the same leading
phase-center moment, multiplied by different known Bessel factors.  Their
generalized least-squares combination therefore yields a closed-form
fixed-frequency variance gain over the endpoint row.  Multiple frequencies
can be incorporated into the same fit, and a second Bessel term can be
retained when deterministic finite-frequency bias becomes significant.
Although multifrequency data are classical in inverse scattering
\cite{Bao2015}, their role here is more specific: they provide repeated
observations of the same phase-center moments and, in the second-order
construction, of their first radial correction.

Cavities expose a different limitation of the phase-center reduction.  Signed
moments detect a cavity whose center is separated from that of the surrounding
material, but a perfectly concentric cavity merges with the material node.
No increase in the number of positive-total-order channels can separate two
coincident phase nodes.  This degeneracy belongs to the phase-center
representation rather than to the full two-angle data.  The complementary
zero-total-order coefficients retain radial Bessel moments.  When the
angularly averaged contrast is piecewise constant in radius, a low-frequency
extrapolation produces a second finite exponential sequence, now in the
squared interface radii.  Its Hankel rank counts radial interfaces and its
shifted pencil recovers their locations.

The contribution of the paper lies in this scattering-specific data reduction
and its stabilization.  We derive the two-angle Bessel--Fourier identity, the
phase-center and radial moment sequences, the exact variance ratio for the
nonnegative total-order channel ensemble, and perturbation conditions for
rank and pencil recovery.  The finite-exponential Hankel factorization and
matrix-pencil step themselves are classical and are used here as algebraic
tools once the scattering data have been reduced to moments.  Related
moment--Hankel ideas have also been used for point-source counting in inverse
heat problems \cite{Deng2026_a,Deng2026_b}; the present setting differs at the
data-to-moment level through the two-angle Bessel--Fourier structure, the
total-order channel redundancy, and the complementary zero-total-order radial
family.

The subsequent contents are organized as follows. Section~\ref{sec:single} develops the two Fourier--Hankel structures.
Section~\ref{sec:multi} treats multichannel estimation, finite-frequency
corrections, and perturbation bounds.  Section~\ref{sec:numerics} presents the
numerical experiments, and Section~\ref{sec:conclusion} concludes the paper.

\section{Fourier--Hankel structures in the two-angle data}
\label{sec:single}

This section develops the two Fourier--Hankel structures used throughout the
paper.  We begin with a Bessel--Fourier representation of the two-angle
far-field coefficients, which makes explicit how the observation and incident
Fourier orders enter the data.  The representation is then reduced in two
different regimes.  Positive total order produces a phase-center moment
sequence whose Hankel structure encodes visible components and their
locations, while zero total order removes the angular phase and produces a
radial moment sequence associated with concentric interfaces.

\subsection{Far-field coefficients and the Bessel--Fourier identity}
\label{subsec:farfield}

Let \(q\in L^\infty(\R^2)\) be compactly supported in a bounded set \(D\).
For an incident plane wave
\[
u^i(x,d;k)=e^{\ii k d\cdot x},
\qquad
d\in\Sone,\quad k>0,
\]
the total field \(u=u^i+u^s\) satisfies
\begin{align}
\Delta u+k^2(1+q)u&=0
\qquad\text{in }\R^2,
\label{eq:helmholtz}
\end{align}
together with the Sommerfeld radiation condition.  With a fixed known
far-field normalization \(C_k\neq0\),
\begin{align}
u^\infty(\hat x,d;k)
&=
C_k\int_D
e^{-\ii k\hat x\cdot y}
q(y)u(y,d;k)\,\dd y.
\label{eq:farfield_exact}
\end{align}
The theoretical reduction below is derived under the Born approximation,
which replaces the total field in the integral by the incident field:
\begin{align}
u^\infty(\hat x,d;k)
&\approx
C_k\int_D
q(y)e^{-\ii k(\hat x-d)\cdot y}\,\dd y.
\label{eq:born_farfield}
\end{align}
The Born approximation is used here to expose a separable angular structure
from which the moment sequences follow.  The stability results below concern
perturbations of those reduced moments and do not constitute a nonlinear
full-wave error analysis.  Section~\ref{subsec:example7} therefore uses
full-wave data only to test the reconstruction outside its derivation regime,
not to verify a theorem beyond the Born model.

Write
\[
\hat x(\theta)=(\cos\theta,\sin\theta),
\qquad
d(\varphi)=(\cos\varphi,\sin\varphi),
\]
and define the two-angle Fourier coefficients
\begin{align}
a_{mn}(k)
&=
\frac{1}{(2\pi)^2}
\int_0^{2\pi}\int_0^{2\pi}
u^\infty(\theta,\varphi;k)
e^{-\ii m\theta}e^{-\ii n\varphi}
\,\dd\theta\,\dd\varphi.
\label{eq:amn}
\end{align}
Here \(m\) is the Fourier order in the observation angle \(\theta\), whereas
\(n\) is the Fourier order in the incident angle \(\varphi\).  These are
Fourier indices, not physical angles.  In particular, \(a_{m0}\) is the
\(m\)-th observation-angle Fourier coefficient after averaging over the
incident angle; it is not data from one fixed incident direction.

The combinations \(m+n\) and \(m-n\) have a useful interpretation.  If
\(\alpha=(\theta+\varphi)/2\) denotes a common rotation and
\(\beta=(\theta-\varphi)/2\) the relative angle, then the Fourier phase can
be written as
\begin{align}
m\theta+n\varphi
&=(m+n)\alpha+(m-n)\beta.
\label{eq:common_relative_indices}
\end{align}
Thus the family \(m+n=p\) is an anti-diagonal in Fourier-index space with a
fixed common-rotation order; it does not impose a relation on the measured
angles themselves.  The special family \(a_{m,-m}\) has zero common-rotation order and selects
a relative-angle harmonic.  These two Fourier families lead to the two moment
reductions developed below.  In the discrete setting we retain orders strictly
below the angular Nyquist index, so that the Fourier coordinates entering the
covariance calculations remain distinct modulo the sampling size.

\begin{proposition}[Bessel--Fourier identity]
\label{prop:bessel-fourier}
Under the Born approximation,
\begin{align}
a_{mn}(k)
&=
C_k(-\ii)^m\ii^n
\int_D
q(r,\psi)
J_m(kr)J_n(kr)
e^{-\ii(m+n)\psi}
r\,\dd r\,\dd\psi.
\label{eq:bessel_fourier}
\end{align}
\end{proposition}

\begin{proof}
The Jacobi--Anger formulas give
\begin{align*}
&e^{-\ii kr\cos(\theta-\psi)}=\sum_{\ell\in\Z}(-\ii)^\ell J_\ell(kr)e^{\ii\ell(\theta-\psi)},\\
&e^{\ii kr\cos(\varphi-\psi)}=\sum_{s\in\Z}\ii^s J_s(kr)e^{\ii s(\varphi-\psi)}.
\end{align*}
Substituting both expansions into \eqref{eq:born_farfield} and using
orthogonality in \(\theta\) and \(\varphi\) selects the indices \(m\) and
\(n\), which yields \eqref{eq:bessel_fourier}.
\end{proof}

Equation \eqref{eq:bessel_fourier} is the basic link between the acoustic
data and the algebraic reconstructions below.  The sum \(m+n\) controls the
angular phase of the contrast, whereas the individual values of \(m\) and
\(n\) determine the radial Bessel kernel.  Fixing \(m+n=p\) therefore leaves
the leading phase-center moment unchanged and creates the multichannel
redundancy used in Section~\ref{sec:multi}.  Setting \(m+n=0\), on the other
hand, removes the angular phase altogether and exposes a complementary radial
moment sequence.

\subsection{Localized components and a single Fourier row}
\label{subsec:single-row}

Assume that the effective contrast is concentrated in \(\Nc\) separated
localized components.  Let
\[
z_j=(x_j, y_j)=r_j(\cos\psi_j,\sin\psi_j),
\qquad
j=1,\ldots,\Nc,
\]
be their phase centers and set
\begin{align}
Q_j=\int_{D_j}q(y)\,\dd y,
\qquad
\lambda_j=r_je^{-\ii\psi_j}=x_j-\ii y_j.
\label{eq:weights_nodes}
\end{align}
The phase-center reduction replaces each localized component by the weighted
point mass \(Q_j\delta_{z_j}\).  Under this approximation,
\begin{align}
a_{mn}^{\rm pc}(k)=
C_k(-\ii)^m\ii^n\sum_{j=1}^{\Nc}Q_jJ_m(kr_j)J_n(kr_j)e^{-\ii(m+n)\psi_j}.
\label{eq:phase_center_coeff}
\end{align}
The finite-size Born coefficient need not equal
\eqref{eq:phase_center_coeff}; their difference is treated as a deterministic
model residual.  The exact rank statement below concerns the reduced
phase-center moments rather than the full finite-size scatterer.

For fixed nonnegative orders, the Bessel product has the expansion
\begin{align}
J_m(t)J_n(t)=\frac{(t/2)^{m+n}}{m!\,n!}
\left[1-\frac{t^2}{4}
\left(\frac{1}{m+1}+\frac{1}{n+1}\right)+O(t^4)
\right].
\label{eq:bessel_product}
\end{align}
When only finitely many orders are retained, the remainder is uniform over
the relevant pairs \((m,n)\).

We first take \(n=0\) and write \(p=m\).  Define
\begin{align}
&c_p^{\rm row}(k)=C_k(-\ii)^p\frac{(k/2)^p}{p!},
\label{eq:row_coefficient}
\\
&M_p=\sum_{j=1}^{\Nc}Q_j\lambda_j^p,
\label{eq:leading_moment}
\\
&R_p=\sum_{j=1}^{\Nc}Q_jr_j^2\lambda_j^p.
\label{eq:radial_moment}
\end{align}
Then \eqref{eq:bessel_product} gives
\begin{align}
\frac{a_{p0}^{\rm pc}(k)}
{c_p^{\rm row}(k)}=M_p-\frac{k^2}{4}
\left(1+\frac{1}{p+1}\right)R_p+O\left(k^4\sum_{j=1}^{\Nc}|Q_j|r_j^{p+4}\right).
\label{eq:row_expansion}
\end{align}

The first term in \eqref{eq:row_expansion} carries the component
information.  It is a finite exponential sequence in the complex phase nodes.
The second term is the leading finite-frequency size correction to the
phase-center moment.  Section~\ref{subsec:multifrequency} retains this term
when several channels and frequencies are combined.

For \(J\geq1\), form the Hankel matrix
\begin{align}
H_J(M)=\big(M_{r+s}\big)_{r, s=0}^{J-1}.
\label{eq:hankel}
\end{align}
In the exact phase-center model, the rank of this matrix follows directly
from a Vandermonde factorization and does not depend on a numerical threshold.

\begin{theorem}[Component count from the row moments]
\label{thm:single_rank}
Assume that the nodes \(\lambda_j\) are pairwise distinct, the weights
\(Q_j\) are nonzero, and \(J\geq\Nc\).  Then
\begin{align}
\rank H_J(M)=\Nc.
\label{eq:single_rank}
\end{align}
\end{theorem}

\begin{proof}
Let
\[
V_J=\big(\lambda_j^r\big)_{\substack{0\leq r\leq J-1\\1\leq j\leq\Nc}}.
\]
The moment definition gives
\begin{align}
H_J(M)=V_J\diag(Q_1,\ldots,Q_{\Nc})V_J^T.
\label{eq:vandermonde_factorization}
\end{align}
Distinct nodes make \(V_J\) full column rank, and the diagonal weight matrix
is nonsingular.
\end{proof}

Once the count has been determined, the same moment sequence yields the
phase-center locations.  With
\begin{align}
H_0=\big(M_{r+s}\big)_{r, s=0}^{\Nc-1},
\quad
H_1=\big(M_{r+s+1}\big)_{r, s=0}^{\Nc-1},
\label{eq:pencil_pair}
\end{align}
the generalized eigenvalues of
\begin{align}
H_1x=\lambda H_0x
\label{eq:pencil}
\end{align}
are \(\lambda_1,\ldots,\lambda_{\Nc}\).  Hence
\begin{align}
z_j=\big(
\operatorname{Re}\lambda_j, -\operatorname{Im}\lambda_j
\big).
\label{eq:center_from_node}
\end{align}
The Hankel rank and shifted-pencil recovery are classical consequences of a
finite exponential sequence
\cite{Hua1990,Potts2010,Potts2013,Peter2013,Kunis2019,CuytLee2024}.
The scattering-specific step is the Bessel--Fourier reduction that produces
this sequence from the measured far-field coefficients.

For later use, we record a perturbation estimate for the pencil itself.  It
applies to both the complex phase-center nodes and the radial nodes introduced
below.

\begin{proposition}[Perturbation of a Hankel pencil]
\label{prop:pencil_perturbation}
Let
\[
u_p=\sum_{j=1}^{r}\omega_j\zeta_j^p,
\qquad
\omega_j\neq 0,
\qquad
\zeta_i\neq\zeta_j\quad(i\neq j),
\]
and form
\[
H_0=(u_{a+b})_{a, b=0}^{r-1},
\qquad
H_1=(u_{a+b+1})_{a, b=0}^{r-1}.
\]
Let
\[
V=(\zeta_j^a)_{\substack{0\leq a\leq r-1\\1\leq j\leq r}},
\qquad
A=H_0^{-1}H_1
=V^{-T}\operatorname{diag}(\zeta_1,\ldots,\zeta_r)V^T.
\]
For \(\widehat H_0=H_0+E_0\) and \(\widehat H_1=H_1+E_1\), assume
\[
\eta_0=\|H_0^{-1}E_0\|_2<1.
\]
Then \(\widehat H_0\) is invertible and, with
\[
\Delta_A=\frac{\|H_0^{-1}\|_2}{1-\eta_0}\left(\|E_1\|_2+\|E_0\|_2\|A\|_2\right),
\]
every generalized eigenvalue \(\widehat\zeta\) of the perturbed pencil
satisfies
\begin{align}
\min_j|\widehat\zeta-\zeta_j|\leq\kappa_2(V)\Delta_A.
\label{eq:pencil_node_bound}
\end{align}
If the right-hand side is smaller than one half of the minimum node
separation, the perturbed nodes can be labeled so that each exact node is
matched with exactly one perturbed node inside its corresponding disk.
\end{proposition}

\begin{proof}
The identity
\[
\widehat A-A=(H_0+E_0)^{-1}(E_1-E_0A)
\]
and the Neumann-series bound give
\[
\|\widehat A-A\|_2\leq\Delta_A.
\]
Since \(A=V^{-T}\operatorname{diag}(\zeta_j)V^T\), the Bauer--Fike theorem
implies \eqref{eq:pencil_node_bound}.  For the last statement, connect the
exact and perturbed pencils by the homotopy
\(H_j(t)=H_j+tE_j\), \(0\leq t\leq1\).  The same Neumann bound keeps
\(H_0(t)\) invertible, and the perturbation disks remain disjoint along the
homotopy.  Continuity of the eigenvalue multiset therefore preserves one
eigenvalue in each disk.
\end{proof}

\subsection{Signed moments and detectable cavities}
\label{subsec:signed}

When cavities are present, the row moments acquire a signed interpretation.
This separates physical presence of a cavity from detectability as a distinct
phase node in the leading moment model.

Suppose
\begin{align}
D=\left(\bigcup_{j=1}^{\Nc}G_j\right)
\setminus
\left(\bigcup_{\ell=1}^{\Nh}B_\ell\right),
\label{eq:signed_geometry}
\end{align}
where \(G_j\) are separated filled material components and each cavity
\(B_\ell\) is compactly contained in one of them.  Let \(q^+\) be the
filled-material contrast and define
\begin{align}
Q_j^+=\int_{G_j}q^+(y)\,\dd y,
\quad
Q_\ell^-=\int_{B_\ell}q^+(y)\,\dd y.
\label{eq:signed_strengths}
\end{align}
With material nodes \(\lambda_j^+\) and cavity nodes \(\lambda_\ell^-\), the
leading signed moments are
\begin{align}
M_p^{\rm sgn}=\sum_{j=1}^{\Nc}Q_j^+(\lambda_j^+)^p-\sum_{\ell=1}^{\Nh}Q_\ell^-(\lambda_\ell^-)^p.
\label{eq:signed_moments}
\end{align}

If a material and cavity node coincide, their contributions combine before
the Hankel matrix is formed.  It is therefore convenient to merge coincident
nodes and delete any term with zero effective weight.  Let the remaining
distinct nodes be
\(\xi_1,\ldots,\xi_{\Ns}\), with nonzero signed weights
\(\gamma_1,\ldots,\gamma_{\Ns}\).  Then
\begin{align}
M_p^{\rm sgn}=\sum_{\nu=1}^{\Ns}\gamma_\nu\xi_\nu^p.
\label{eq:effective_signed}
\end{align}

\begin{theorem}[Signed Hankel rank]
\label{thm:signed_rank}
If \(J\geq\Ns\), the effective nodes are pairwise distinct, and their weights
are nonzero, then
\begin{align}
\rank H_J(M^{\rm sgn})=\Ns.
\label{eq:signed_rank}
\end{align}
If all material and cavity phase nodes are pairwise distinct, then
\(\Ns=\Nc+\Nh\), and therefore
\begin{align}
\Nh=\rank H_J(M^{\rm sgn})-\Nc.
\label{eq:cavity_count}
\end{align}
\end{theorem}

\begin{proof}
The signed sequence \eqref{eq:effective_signed} has the factorization
\[
H_J(M^{\rm sgn})=V_{\rm sgn}\diag(\gamma_1,\ldots,\gamma_{\Ns})V_{\rm sgn}^T,
\]
and the conclusion follows as in Theorem~\ref{thm:single_rank}.
\end{proof}

\begin{remark}[Coincident-node degeneracy]
\label{rem:detectability}
If a material and cavity phase center coincide, their leading signed
contributions merge into one effective node.  The signed phase-center rank
therefore counts detectable cavity phase centers rather than geometric holes
without qualification.  Using more coefficients with the same positive total
angular order reduces noise in the common signed moment, but it cannot create
a second phase node when the two centers coincide.  The next subsection shows
that the missing information is nevertheless present in a different part of
the two-angle Fourier matrix.
\end{remark}

\subsection{Zero-total-order coefficients and radial interfaces}
\label{subsec:radial}

The signed phase-center model loses a concentric cavity because material and
cavity share the same angular phase.  The full far-field matrix nevertheless
contains complementary information, obtained here from the zero-total-order
Fourier coefficients.

Define the angular average of the contrast about the chosen origin by
\begin{align}
q_0(r)=\frac{1}{2\pi}\int_0^{2\pi}q(r,\psi)\,\dd\psi.
\label{eq:radial_average}
\end{align}

\begin{proposition}[Zero-total-order radial identity]
\label{prop:zero_total_radial}
Under the Born approximation, for every nonnegative integer \(m\),
\begin{align}
a_{m,-m}(k)=a_{-m,m}(k)=2\pi C_k\int_0^\infty q_0(r)J_m(kr)^2r\,\dd r.
\label{eq:zero_total_radial}
\end{align}
If the contrast is radially symmetric about the origin, then
\begin{align}
a_{mn}(k)=0
\qquad\text{whenever }m+n\neq0.
\label{eq:radial_positive_total_vanish}
\end{align}
Hence no collection of positive-total-order channels, at one or many
frequencies, can create a second phase node for a perfectly concentric radial
profile.
\end{proposition}

\begin{proof}
Setting \(n=-m\) in Proposition~\ref{prop:bessel-fourier} and using
\(J_{-m}(t)=(-1)^mJ_m(t)\) cancels the Fourier phase and all prefactors,
which gives \eqref{eq:zero_total_radial}.  If \(q=q(r)\), the remaining
angular integral in \eqref{eq:bessel_fourier} is the integral of
\(e^{-\ii(m+n)\psi}\) over a full period and therefore vanishes for
\(m+n\neq0\).
\end{proof}

\begin{remark}[Scope of the radial reduction]
\label{rem:radial_scope}
The coefficients \(a_{m,-m}\) determine radial Bessel moments of the angular
average \(q_0\).  They therefore resolve concentric radial interfaces when
that average has a finite layered representation.  For a general nonradial
cavity with the same phase center as the surrounding material, the present
radial rank describes the interfaces of \(q_0\), not the full cavity shape.
Thus the radial construction removes the concentric-annulus degeneracy without
claiming unconditional recovery of arbitrary same-centered holes.
\end{remark}

Assume that \(q_0\) is piecewise constant and write
\begin{align}
q_0(r)=\sum_{\ell=1}^{L}\beta_\ell\mathbf 1_{[0,R_\ell)}(r),
\qquad 0<R_1<\cdots<R_L.
\label{eq:radial_disk_superposition}
\end{align}
The coefficients \(\beta_\ell\) are the jumps of the radial profile and may
have either sign.  Let \(R_*\geq R_L\) be a reference radius, chosen for
example from the imaging window.  The normalization by \(R_*\) is not needed
for the exact rank, but keeps the radial nodes in \((0,1]\) and improves
numerical conditioning.

Define
\begin{align}
s_m(k)=\frac{(2m+2)(m!)^2}{2\pi C_kR_*^{2m}}\left(\frac{2}{k}\right)^{2m}a_{m,-m}(k),
\label{eq:scaled_radial_coeff}
\end{align}
and
\begin{align}
&\mu_m^{\rm rad}=\sum_{\ell=1}^{L}\beta_\ell R_\ell^2\left(\frac{R_\ell^2}{R_*^2}\right)^m
=\sum_{\ell=1}^{L}\gamma_\ell\xi_\ell^m,
\label{eq:radial_moment_sequence}
\\
&\xi_\ell=\frac{R_\ell^2}{R_*^2},
\qquad \gamma_\ell=\beta_\ell R_\ell^2.
\nonumber
\end{align}
For later error bounds set
\[
A_m^{\rm rad}=\sum_{\ell=1}^{L}|\gamma_\ell|\xi_\ell^m.
\]

\begin{proposition}[Low-frequency radial expansion]
\label{prop:radial_expansion}
Fix a maximum retained order.  As \(kR_*\to0\), uniformly over the retained
orders,
\begin{align}
&s_m(k)=\mu_m^{\rm rad}-\frac{(kR_*)^2}{2(m+2)}\mu_{m+1}^{\rm rad}
\nonumber\\
&+\frac{(2m+3)(kR_*)^4}{16(m+1)(m+2)(m+3)}\mu_{m+2}^{\rm rad}+O\left((kR_*)^6A_{m+3}^{\rm rad}\right).
\label{eq:radial_low_frequency_expansion}
\end{align}
\end{proposition}

\begin{proof}
The Bessel series gives
\[
J_m(t)^2=\frac{(t/2)^{2m}}{(m!)^2}\left[1-\frac{t^2}{2(m+1)}+\frac{(2m+3)t^4}{16(m+1)^2(m+2)}+O(t^6)\right].
\]
Insert this expression into \eqref{eq:zero_total_radial}, integrate each disk
term in \eqref{eq:radial_disk_superposition}, and apply the normalization in
\eqref{eq:scaled_radial_coeff}.  The remainder is uniform because only
finitely many Bessel orders are retained and \(R_\ell\leq R_*\).
\end{proof}

The zero-frequency limit is a finite exponential sequence in normalized
squared radii.

\begin{theorem}[Radial interface rank and recovery]
\label{thm:radial_rank}
Assume that the radii in \eqref{eq:radial_disk_superposition} are distinct,
all \(\gamma_\ell\) are nonzero, and \(J\geq L\).  The radial Hankel matrix
\begin{align}
H_J^{\rm rad}=\big(\mu_{r+s}^{\rm rad}\big)_{r, s=0}^{J-1}
\label{eq:radial_hankel}
\end{align}
satisfies
\begin{align}
\rank H_J^{\rm rad}=L.
\label{eq:radial_rank}
\end{align}
For the square \(L\times L\) shifted pencil, the generalized eigenvalues are
\(\xi_1,\ldots,\xi_L\), and
\begin{align}
R_\ell=R_*\sqrt{\xi_\ell}.
\label{eq:radial_radius_from_node}
\end{align}
\end{theorem}

\begin{proof}
With \(V_J^{\rm rad}=(\xi_\ell^r)\),
\begin{align}
H_J^{\rm rad}=V_J^{\rm rad}\diag(\gamma_1,\ldots,\gamma_L)(V_J^{\rm rad})^T.
\end{align}
The Vandermonde matrix has full column rank.  The shifted factorization adds
one power of each \(\xi_\ell\), so the generalized eigenvalues are exactly
the normalized squared radii.
\end{proof}

For a filled disk there is one radial node.  For a concentric annulus with
inner radius \(\rho\), outer radius \(R\), and constant contrast \(q_0\),
\begin{align}
\mu_m^{\rm rad}=q_0R^2\left(\frac{R^2}{R_*^2}\right)^m-q_0\rho^2\left(\frac{\rho^2}{R_*^2}\right)^m,
\label{eq:annulus_radial_sequence}
\end{align}
so the radial rank is two although the phase-center rank is one.

The construction can also be applied after recentering a component.  If its
phase center is \(z_0\), then under the Born model
\begin{align}
\widetilde u^\infty(\hat x,d;k)=e^{\ii k(\hat x-d)\cdot z_0}u^\infty(\hat x,d;k)
\label{eq:recenter_farfield}
\end{align}
is the far field of the translated contrast.  If \(z_0\) is replaced by
\(\widehat z_0\), then
\[
\left|e^{\ii k(\hat x-d)\cdot\widehat z_0}-e^{\ii k(\hat x-d)\cdot z_0}\right|\leq2k|\widehat z_0-z_0|,
\]
so centering error enters the radial moments as an additional deterministic
perturbation.
\section{Multichannel estimation and stability}
\label{sec:multi}

The preceding section identifies two finite moment structures in the same
two-angle data.  We now study how measurement noise and finite-frequency bias
enter their empirical Hankel matrices.  The phase-center sequence benefits
from redundancy along fixed-total-order anti-diagonals.  The radial sequence
uses the complementary zero-total-order family and a low-frequency
extrapolation.

\subsection{Noise transferred by a single channel}
\label{subsec:row-noise}

Suppose the far field at a fixed frequency is sampled on the uniform grid
\[
\theta_r=\frac{2\pi r}{\Nang},
\qquad
\varphi_s=\frac{2\pi s}{\Nang}.
\]
The discrete coefficient is
\begin{align}
a_{mn}^{(\Nang)}=\frac{1}{\Nang^2}\sum_{r, s=0}^{\Nang-1}U_{r s}^{\infty}e^{-2\pi\ii mr/\Nang}e^{-2\pi\ii ns/\Nang}.
\label{eq:discrete_fourier}
\end{align}
We perturb the full far-field matrix according to
\begin{align}
U_\delta^\infty=U^\infty+\delta\|U^\infty\|_F\frac{W}{\|W\|_F},
\label{eq:noise_model}
\end{align}
where the entries of \(W\) are independent circular complex Gaussian
variables.  This is a relative Frobenius perturbation of the measured
far-field matrix rather than a perturbation added directly to the moments.

\begin{proposition}[Covariance of the Fourier coefficients]
\label{prop:fourier_noise}
Let \(\eta_{mn}\) denote the perturbation of
\eqref{eq:discrete_fourier} induced by \eqref{eq:noise_model}.  Then
\begin{align}
&\mathbb E\eta_{mn}=0,
\\
&\mathbb E\left[\eta_{mn}\overline{\eta_{m'n'}}\right]=\sigma_k^2\delta_{mm'}\delta_{nn'},
\label{eq:fourier_cov}
\end{align}
where
\begin{align}
\sigma_k^2=\frac{\delta^2\|U^\infty\|_F^2}{\Nang^4}.
\label{eq:sigma_k}
\end{align}
\end{proposition}

\begin{proof}
After vectorization, rotational invariance of the normalized Gaussian vector
gives
\[
\mathbb E\frac{\operatorname{vec}(W)\operatorname{vec}(W)^*}{\|W\|_F^2}
=\frac{1}{\Nang^2}I.
\]
The two-dimensional discrete Fourier transform is unitary up to the
deterministic scale \(\Nang\), while
\eqref{eq:discrete_fourier} contributes \(1/\Nang^2\).  This yields
\eqref{eq:fourier_cov}.
\end{proof}

The empirical row moment is
\begin{align}
\widehat M_p^{\rm row}=\frac{a_{p0}^{(\Nang)}}{c_p^{\rm row}(k)}.
\label{eq:row_estimator}
\end{align}
If the deterministic phase-center and quadrature errors are set aside, its
variance is therefore
\begin{align}
\Var\left(\widehat M_p^{\rm row}\right)
=\frac{\sigma_k^2}{|C_k|^2}(p!)^2\left(\frac{2}{k}\right)^{2p}.
\label{eq:row_variance}
\end{align}
This formula isolates the numerical difficulty of the one-row reduction.  The
factorial amplification occurs before the Hankel matrix is assembled; changing
the rank algorithm alone does not remove it.

\subsection{Total-order channels}
\label{subsec:total-order}

The two-angle Fourier matrix contains more information about the same leading
moment.  For a nonnegative total order \(p\), let
\[
\mathcal A_p=\{(m, n)\in\mathbb N_0^2: m+n=p\}.
\]
For \((m, n)\in\mathcal A_p\), define
\begin{align}
c_{mn}^{(p)}(k)&=C_k(-\ii)^m\ii^n
\frac{(k/2)^p}{m!\,n!},
\label{eq:cmn}
\\
d_{mn}(k)&=\frac{k^2}{4}
\left(\frac{1}{m+1}+\frac{1}{n+1}
\right).
\label{eq:dmn}
\end{align}
Equation \eqref{eq:bessel_product} gives
\begin{align}
a_{mn}^{\rm pc}(k)=c_{mn}^{(p)}(k)\left[M_p-d_{mn}(k)R_p+\rho_{mn,p}(k)\right],
\qquad m+n=p,
\label{eq:all_channel_model}
\end{align}
with
\begin{align}
|\rho_{mn,p}(k)|\leq
C_Pk^4\sum_{j=1}^{\Nc}|Q_j|r_j^{p+4}
\label{eq:rho_bound}
\end{align}
for any fixed finite range of retained orders.

In \eqref{eq:all_channel_model}, the leading quantity \(M_p\) is
independent of how \(p\) is split into \(m+n\).  Thus the \(p+1\)
coefficients indexed by \(\mathcal A_p\) provide repeated measurements of
the same scalar parameter.  At a fixed frequency the coefficient errors have
equal variance, so the least-squares estimate is
\begin{align}
\widehat M_p^{\rm all}=\frac{
\sum_{m=0}^{p}\overline{c_{m, p-m}^{(p)}}\,a_{m, p-m}^{(\Nang)}}{\sum_{m=0}^{p}|c_{m, p-m}^{(p)}|^2}.
\label{eq:all_estimator}
\end{align}

\begin{theorem}[Variance ratio at one frequency]
\label{thm:variance_ratio}
Under the leading model
\[
a_{m, p-m}^{(\Nang)}=c_{m, p-m}^{(p)}M_p+\eta_{m, p-m},
\]
the row estimator \eqref{eq:row_estimator} and the total-order estimator
\eqref{eq:all_estimator} satisfy
\begin{align}
\frac{\Var(\widehat M_p^{\rm row})}{\Var(\widehat M_p^{\rm all})}=\binom{2p}{p}.
\label{eq:variance_ratio}
\end{align}
Equivalently,
\begin{align}
\frac{\operatorname{sd}(\widehat M_p^{\rm row})}{\operatorname{sd}(\widehat M_p^{\rm all})}=
\sqrt{\binom{2p}{p}}.
\label{eq:sd_ratio}
\end{align}
\end{theorem}

\begin{proof}
At fixed \(k\),
\[
\frac{|c_{m,p-m}^{(p)}|}
{|c_{p,0}^{(p)}|}=\frac{p!}{m!(p-m)!}=
\binom{p}{m}.
\]
The Fourier coefficient errors are uncorrelated with common variance, and
hence
\[
\frac{\Var(\widehat M_p^{\rm row})}{\Var(\widehat M_p^{\rm all})}=
\frac{\sum_{m=0}^{p}|c_{m,p-m}^{(p)}|^2}{|c_{p,0}^{(p)}|^2}=
\sum_{m=0}^{p}\binom{p}{m}^2=\binom{2p}{p}.
\]
\end{proof}

For the largest moment order used in the numerical experiments, \(p=8\),
the standard-deviation ratio in \eqref{eq:sd_ratio} is approximately
\(113.45\).  This reduction uses the same frequency and far-field matrix; it comes solely
from retaining angular information discarded by the row construction.

The same conclusion holds for the signed moments in
\eqref{eq:signed_moments}.  Indeed, changing the signs of the component
weights changes \(M_p\) but not the channel coefficients or the Fourier-noise
covariance.

\begin{corollary}[Signed moments]
\label{cor:signed_variance}
Under the signed phase-center model, the total-order estimator of
\(M_p^{\rm sgn}\) satisfies
\begin{align}
\frac{\Var(\widehat M_{p,\rm row}^{\rm sgn})}{\Var(\widehat M_{p,\rm all}^{\rm sgn})}=\binom{2p}{p}.
\label{eq:signed_variance_ratio}
\end{align}
Thus the same multichannel variance reduction applies to the moments used for
cavity detection.
\end{corollary}

\subsection{Several frequencies and the first radial correction}
\label{subsec:multifrequency}

A single frequency is sufficient for Theorem~\ref{thm:variance_ratio}, but
several frequencies provide additional measurements of the same phase nodes.
Let \(k_1,\ldots,k_K\) be measured frequencies and collect all
frequency--channel pairs with \(m+n=p\).  For an index \(i=(\ell,m,n)\), write
\begin{align}
y_i=c_iM_p+r_i+\eta_i,
\label{eq:multi_linear_model}
\end{align}
where \(r_i\) contains deterministic Bessel, finite-size, and quadrature
errors, while
\[
\Var(\eta_i)=\sigma_i^2.
\]
Define
\begin{align}
\mathcal J_p=\sum_i\frac{|c_i|^2}{\sigma_i^2}.
\label{eq:information}
\end{align}
The generalized least-squares estimate is
\begin{align}
\widehat M_p^{(0)}=
\mathcal J_p^{-1}\sum_i\frac{\overline{c_i}y_i}{\sigma_i^2}.
\label{eq:gls}
\end{align}

\begin{proposition}[Bias and variance of the multichannel estimate]
\label{prop:moment_covariance}
For \eqref{eq:multi_linear_model},
\begin{align}
\mathbb E\widehat M_p^{(0)}-M_p
&=\mathcal J_p^{-1}\sum_i\frac{\overline{c_i}r_i}{\sigma_i^2},
\label{eq:gls_bias}
\\
\mathbb E\left|\widehat M_p^{(0)}-\mathbb E\widehat M_p^{(0)}\right|^2
&=\mathcal J_p^{-1}.
\label{eq:gls_variance}
\end{align}
\end{proposition}

\begin{proof}
Take the expectation of \eqref{eq:gls} and use the diagonal Fourier-noise
covariance.  The variance of the weighted numerator is
\(\mathcal J_p\), which gives \eqref{eq:gls_variance}.
\end{proof}

Adding independent frequency measurements increases the information sum
\(\mathcal J_p\).  Using larger frequencies, however, also weakens the
small-argument Bessel approximation.  We therefore retain the first radial
correction in \eqref{eq:all_channel_model}.  The resulting local model is
\begin{align}
y_i=c_i M_p-c_i d_i R_p+\widetilde r_i+\eta_i.
\label{eq:second_order_model}
\end{align}
With
\[
X_p=
\begin{bmatrix}
c_i & -c_id_i
\end{bmatrix}_i,
\qquad
\Sigma_p=\diag(\sigma_i^2)_i,
\]
set
\begin{align}
\widehat\theta_p=
(X_p^*\Sigma_p^{-1}X_p)^{-1}X_p^*\Sigma_p^{-1}y_p,
\qquad
\theta_p=
\begin{bmatrix}
M_p\\R_p
\end{bmatrix}.
\label{eq:second_gls}
\end{align}
The first component is denoted by \(\widehat M_p^{(2)}\).

\begin{proposition}[Second-order Bessel correction]
\label{prop:second_order}
Suppose \(k_\ell=s\kappa_\ell\), the retained orders are fixed, and the
normalized two-column design in \eqref{eq:second_gls} remains nondegenerate
as \(s\to0\).  For noiseless phase-center data,
\begin{align}
\widehat M_p^{(0)}-M_p
&=
O(s^2),
\label{eq:O2}
\\
\widehat M_p^{(2)}-M_p
&=
O(s^4).
\label{eq:O4}
\end{align}
Moreover,
\begin{align}
\operatorname{Cov}(\widehat\theta_p)
&=
(X_p^*\Sigma_p^{-1}X_p)^{-1}.
\label{eq:second_cov}
\end{align}
\end{proposition}

\begin{proof}
After division by \(c_i\), equation
\eqref{eq:all_channel_model} becomes
\[
M_p-s^2\widetilde d_iR_p+O(s^4).
\]
A one-parameter fit omits the quadratic term and therefore has a
second-order intercept bias.  The two-parameter fit represents that term
explicitly.  The normalized-design assumption prevents the fourth-order
remainder from being amplified in the recovered intercept.  The covariance
formula is the standard generalized least-squares identity.
\end{proof}

The second-order estimate is intended for a different numerical regime from
the leading estimate.  It is useful when deterministic Bessel bias
limits localization, but the additional radial parameter increases variance.
The numerical experiments below illustrate this bias--variance trade-off.

\subsection{Effect on the Hankel rank}
\label{subsec:rank-stability}

We next connect the moment error to the numerical Hankel rank.  Write
\begin{align}
\widehat M_p
&=
M_p+\beta_p+\zeta_p,
\qquad
\mathbb E\zeta_p=0,
\qquad
\mathbb E|\zeta_p|^2=v_p,
\label{eq:moment_error}
\end{align}
where \(\beta_p\) collects deterministic model and discretization errors.
For the leading generalized least-squares estimator,
\[
v_p=\mathcal J_p^{-1}.
\]
The error \(\zeta_p\) appears several times on the \(p\)-th anti-diagonal of
the Hankel matrix.  Define
\begin{align}
\nu_p
&=
\begin{cases}
p+1, & 0\leq p\leq J-1,\\
2J-1-p, & J\leq p\leq2J-2,
\end{cases}
\label{eq:multiplicity}
\end{align}
and set
\begin{align}
\mathcal V_J
&=
\sum_{p=0}^{2J-2}
\nu_pv_p.
\label{eq:hankel_noise_energy}
\end{align}
Then \(\mathcal V_J\) is the expected squared Frobenius norm of the stochastic
Hankel perturbation.

\begin{theorem}[Hankel rank stability]
\label{thm:rank_stability}
Let an exact Hankel matrix \(H_J\) have rank \(r\), with
\[
s_*=\sigma_r(H_J)>0,
\qquad J\geq r+1.
\]
Suppose
\[
\widehat H_J=H_J+B_J+Z_J,
\qquad \mathbb E\|Z_J\|_F^2=\mathcal V_J.
\]
If
\begin{align}
\|B_J\|_2+t&<\frac12s_*,
\label{eq:rank_condition}
\end{align}
then, with probability at least \(1-\mathcal V_J/t^2\), the threshold rank of
\(\widehat H_J\) is \(r\) for every absolute threshold satisfying
\begin{align}
\|B_J\|_2+t<\tau_{\rm abs}<s_*-\|B_J\|_2-t.
\label{eq:threshold_interval}
\end{align}
\end{theorem}

\begin{proof}
Since \(\|Z_J\|_2\leq\|Z_J\|_F\), Markov's inequality gives
\[
\mathbb P(\|Z_J\|_2\geq t)\leq\frac{\mathcal V_J}{t^2}.
\]
On the complementary event,
\[
\|\widehat H_J-H_J\|_2\leq\|B_J\|_2+t.
\]
Weyl's inequality separates the last signal singular value from the first
tail singular value under \eqref{eq:rank_condition}.
\end{proof}

For the phase-center moments in \eqref{eq:moment_error},
\(B_J=(\beta_{r+s})\) and \(\mathcal V_J\) is given by
\eqref{eq:hankel_noise_energy}.  More total-order channels reduce the
stochastic perturbation but do not change the exact singular gap.

\subsection{Radial extrapolation and stability}
\label{subsec:radial_stability}

For \(m\geq1\), the exact radial signals in \(a_{m,-m}\) and \(a_{-m,m}\)
are equal, so we average them; for \(m=0\) we use \(a_{00}\).  For retained
orders below the angular Nyquist index these are distinct Fourier coordinates.
Proposition~\ref{prop:fourier_noise} therefore shows that the averaging halves
the coefficient variance for \(m\geq1\).

Let \(k_1,\ldots,k_K\) be measured frequencies, with independent noise
realizations across frequencies, and put
\[
x_\ell=(k_\ell R_*)^2.
\]
After the scaling in \eqref{eq:scaled_radial_coeff}, write
\begin{align}
y_{m\ell}
&=\mu_m^{\rm rad}+b_mx_\ell+r_{m\ell}+\varepsilon_{m\ell},
\qquad b_m=-\frac{\mu_{m+1}^{\rm rad}}{2(m+2)}.
\label{eq:radial_regression_model}
\end{align}
For fixed retained orders,
\begin{align}
|r_{m\ell}|&\leq C_Mx_\ell^2A_{m+2}^{\rm rad}.
\label{eq:radial_regression_remainder}
\end{align}
If \(g_{m\ell}\) is the deterministic scaling factor in
\eqref{eq:scaled_radial_coeff}, then
\begin{align}
\tau_{m\ell}^2:=\operatorname{Var}(\varepsilon_{m\ell})
&=\chi_m|g_{m\ell}|^2\sigma_{k_\ell}^2,
\qquad \chi_0=1,
\quad \chi_m=\frac12\quad(m\geq1).
\label{eq:radial_noise_variance}
\end{align}

Let \(X\in\R^{K\times2}\) have rows \((1,x_\ell)\), let
\(\Sigma_m=\diag(\tau_{m1}^2,\ldots,\tau_{mK}^2)\), and define
\begin{align}
\widehat\vartheta_m
&=(X^*\Sigma_m^{-1}X)^{-1}X^*\Sigma_m^{-1}y_m,
\qquad \widehat\mu_m^{\rm rad}=e_1^T\widehat\vartheta_m.
\label{eq:radial_gls}
\end{align}

\begin{proposition}[Bias and variance of the radial extrapolation]
\label{prop:radial_gls}
Let
\[
w_m^T=e_1^T(X^*\Sigma_m^{-1}X)^{-1}X^*\Sigma_m^{-1}.
\]
Then
\begin{align}
\mathbb E\widehat\mu_m^{\rm rad}-\mu_m^{\rm rad}&=w_m^Tr_m,
\label{eq:radial_gls_bias}
\\
\operatorname{Var}(\widehat\mu_m^{\rm rad})
&=e_1^T(X^*\Sigma_m^{-1}X)^{-1}e_1=:v_m^{\rm rad},
\label{eq:radial_gls_variance}
\end{align}
and
\begin{align}
\left|\mathbb E\widehat\mu_m^{\rm rad}-\mu_m^{\rm rad}\right|
&\leq C_M\|w_m\|_1x_{\max}^2A_{m+2}^{\rm rad}.
\label{eq:radial_bias_bound}
\end{align}
\end{proposition}

\begin{proof}
The weighted fit reproduces the constant and linear terms, so
\(w_m^T\mathbf1=1\) and \(w_m^Tx=0\).  The bias therefore contains only the
remainder.  The covariance identity follows from
\eqref{eq:radial_noise_variance}.
\end{proof}

Write
\[
\widehat\mu_m^{\rm rad}=\mu_m^{\rm rad}+\beta_m^{\rm rad}+\zeta_m^{\rm rad},
\qquad \mathbb E|\zeta_m^{\rm rad}|^2=v_m^{\rm rad}.
\]
For
\[
b_m^{\rm det}=C_M\|w_m\|_1x_{\max}^2A_{m+2}^{\rm rad},
\]
Proposition~\ref{prop:radial_gls} gives
\(|\beta_m^{\rm rad}|\leq b_m^{\rm det}\).  Hence, with
\begin{align}
\mathcal B_J^{\rm rad}
&=\left(\sum_{m=0}^{2J-2}\nu_m(b_m^{\rm det})^2\right)^{1/2},
\label{eq:radial_hankel_bias}
\\
\mathcal V_J^{\rm rad}
&=\sum_{m=0}^{2J-2}\nu_mv_m^{\rm rad},
\label{eq:radial_hankel_variance}
\end{align}
we have
\(\|B_J^{\rm rad}\|_2\leq\mathcal B_J^{\rm rad}\) and
\(\mathbb E\|Z_J^{\rm rad}\|_F^2=\mathcal V_J^{\rm rad}\).
Here \(\nu_m\) is given by \eqref{eq:multiplicity}.

\begin{corollary}[Stable radial interface count]
\label{cor:radial_rank_stability}
Let \(s_*^{\rm rad}=\sigma_L(H_J^{\rm rad})\), with \(J\geq L+1\).
For a confidence parameter \(0<\alpha<1\), set
\begin{align}
\varepsilon_{J,\alpha}^{\rm rad}
&=\mathcal B_J^{\rm rad}
+\sqrt{\frac{\mathcal V_J^{\rm rad}}{\alpha}}.
\label{eq:radial_confidence_radius}
\end{align}
If
\begin{align}
\varepsilon_{J,\alpha}^{\rm rad}
&<\frac12s_*^{\rm rad},
\label{eq:radial_rank_condition}
\end{align}
then, with probability at least \(1-\alpha\), every absolute threshold in
\begin{align}
\varepsilon_{J,\alpha}^{\rm rad}
<\tau_{\rm abs}<
 s_*^{\rm rad}-\varepsilon_{J,\alpha}^{\rm rad}
\label{eq:radial_threshold_window}
\end{align}
returns exactly \(L\) radial interfaces.  On the same event, the first tail
singular value satisfies
\(\sigma_{L+1}(\widehat H_J^{\rm rad})\leq
\varepsilon_{J,\alpha}^{\rm rad}\), irrespective of the size of the signal
gap.
\end{corollary}

\begin{proof}
The deterministic bound follows from
\eqref{eq:radial_bias_bound} and the anti-diagonal multiplicities.  Markov's
inequality gives
\[
\mathbb P\left(\|Z_J^{\rm rad}\|_2>
\sqrt{\mathcal V_J^{\rm rad}/\alpha}\right)\leq\alpha.
\]
Apply Theorem~\ref{thm:rank_stability} on the complementary event.
\end{proof}

The radial pencil is covered by Proposition~\ref{prop:pencil_perturbation}.
For real positive exact nodes, suppose a matched estimate satisfies
\(|\widehat\xi_\ell-\xi_\ell|<\xi_\ell\) and is projected to the positive
real axis when a small imaginary part is produced by noise.  Then the radius
map is locally Lipschitz; in particular, for positive matched nodes,
\[
|\widehat R_\ell-R_\ell|
=R_*\frac{|\widehat\xi_\ell-\xi_\ell|}
{\sqrt{\widehat\xi_\ell}+\sqrt{\xi_\ell}}
\leq
\frac{R_*}{\sqrt{\xi_\ell}}
|\widehat\xi_\ell-\xi_\ell|.
\]
Thus close radial interfaces and small inner radii increase the sensitivity of
the recovered radii, just as close phase centers increase the conditioning of
the phase-center pencil.

\section{Numerical experiments}
\label{sec:numerics}

The experiments test the mechanisms isolated in Sections~\ref{sec:single}
and~\ref{sec:multi}.  Unless stated otherwise, the far field is sampled on a
\(128\times128\) uniform grid of observation and incident directions.
Examples~\ref{subsec:example1}--\ref{subsec:example6} use analytic Born data,
so their reported errors contain the finite-size Bessel effect but no
volume-quadrature error.  Example~\ref{subsec:example7} is different: its data
are generated by a full-wave Helmholtz solve and the Born approximation is
used only in the reconstruction model.  Noise is added to the full far-field
matrix through \eqref{eq:noise_model} when specified.
For a Hankel matrix \(H\), we use the diagnostic relative rank
\begin{align}
\operatorname{rank}_{\tau}(H)=
\#\left\{j:\frac{\sigma_j(H)}{\sigma_1(H)}>\tau\right\},
\qquad \tau=10^{-3},
\label{eq:numerical-rank}
\end{align}
only when a fixed threshold is useful for comparison.  Section~\ref{subsec:example6}
shows why this value should not be interpreted as a universal rank rule.
Throughout this section, ``all channels'' for total order \(p\) means the
\(p+1\) nonnegative pairs \((m,n)\) with \(m+n=p\) used in
Theorem~\ref{thm:variance_ratio}.

\subsection{Component counting and total-order stabilization}
\label{subsec:example1}

We first consider \(N_c=1,2,3,4\) identical unit-contrast disks of radius
\(0.025\), with centers
\begin{align}
N_c=1:\quad&(0,0),\nonumber\\
N_c=2:\quad&(-0.30,-0.10),\ (0.31,0.13),\nonumber\\
N_c=3:\quad&(-0.37,-0.18),\ (0.02,0.33),\ (0.36,-0.15),\nonumber\\
N_c=4:\quad&(-0.38,-0.25),\ (-0.20,0.30),\
(0.23,-0.31),\ (0.38,0.23).
\label{eq:example1-geometries}
\end{align}
At \(k=2\) we form a \(5\times5\) Hankel matrix, requiring moments through total
order \(p=8\).  The single-row estimate uses \(a_{p0}\) only.  Figure~\ref{fig:example1-clean}
shows the corresponding clean spectra.  The numerical ranks are
\(1,2,3,4\); for the four-disk case,
\begin{align}
\frac{\sigma_4}{\sigma_1}=2.093\times10^{-2},
\qquad
\frac{\sigma_5}{\sigma_1}=1.302\times10^{-4}.
\label{eq:example1-clean-gap}
\end{align}
Thus the finite disk size perturbs, but does not obscure, the rank-four
phase-center structure in the clean data.

\begin{figure}[htbp]
\centering
\includegraphics[width=0.72\textwidth]{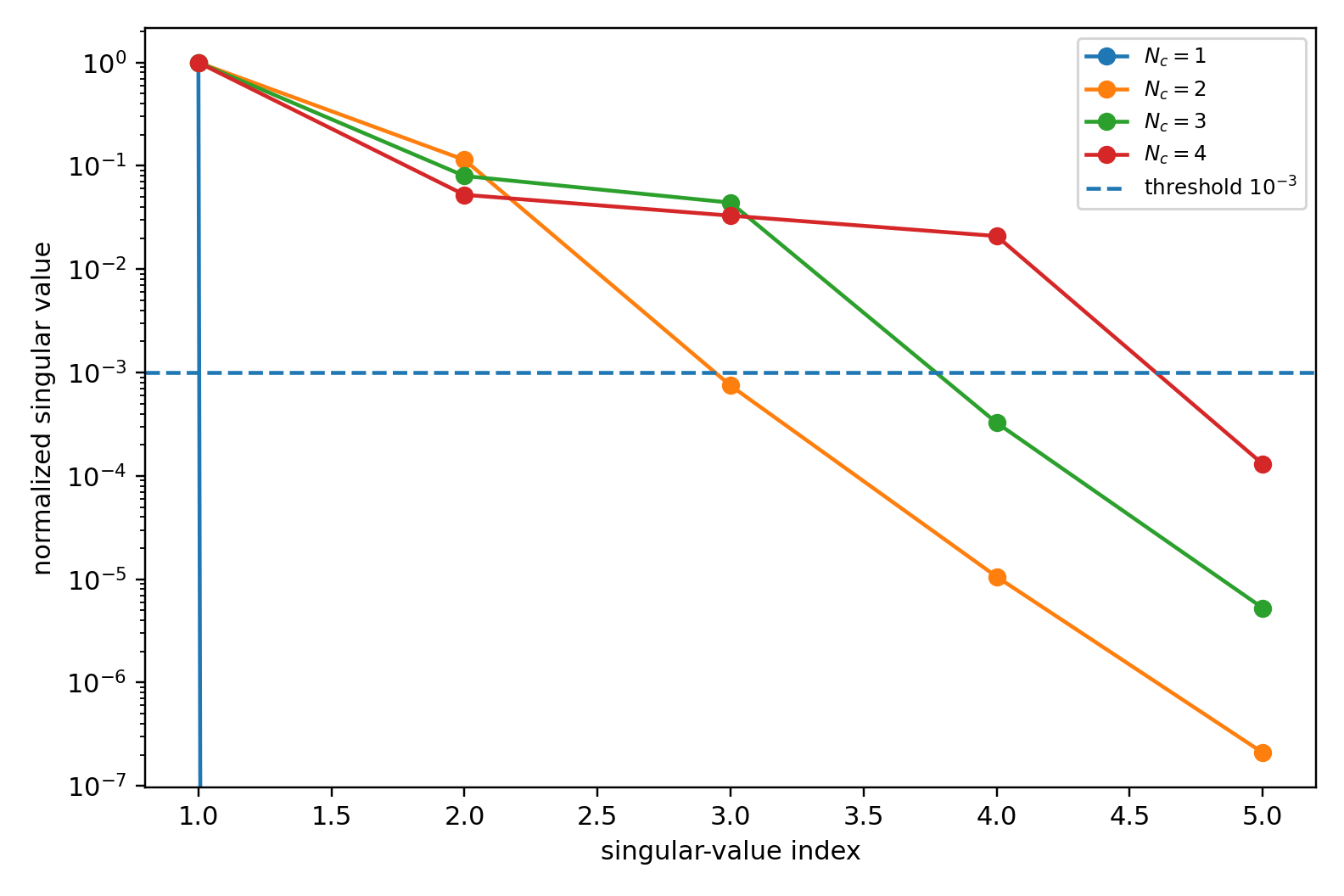}
\caption{Normalized single-row Hankel spectra for the four clean
configurations.  The dashed line is the reference threshold \(10^{-3}\).}
\label{fig:example1-clean}
\end{figure}

We next compare three noisy estimators: the single row at \(k=2\), all
nonnegative total-order channels at \(k=2\), and the same channels at
\(k=1.2,1.4,1.6,1.8,2.0,2.2,2.4\).  For each component count and noise level,
120 independent perturbations are generated.  Figure~\ref{fig:example1-accuracy}
reports the worst correct-count probability over \(N_c=1,\ldots,4\).  The
single-row method falls from perfect recovery at relative noise \(10^{-6}\) to
\(0.067\) at \(10^{-5}\) and zero at \(10^{-4}\).  The fixed-frequency
multichannel estimator remains perfect through \(10^{-4}\), while the
seven-frequency estimator is perfect through \(3\times10^{-4}\) and retains a
worst success probability of \(0.942\) at \(10^{-3}\).

\begin{figure}[htbp]
\centering
\includegraphics[width=0.74\textwidth]{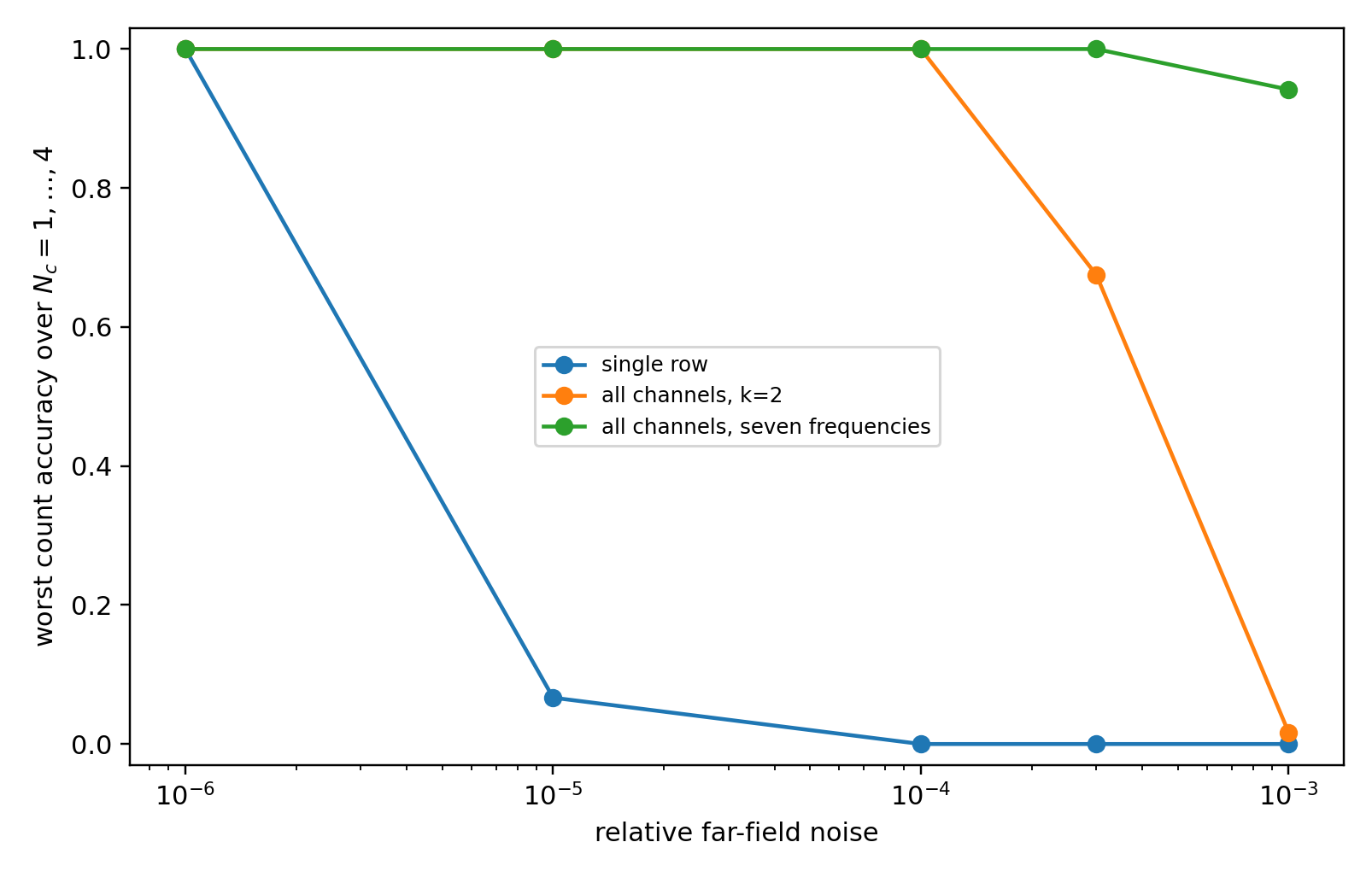}
\caption{Worst component-count accuracy over \(N_c=1,\ldots,4\) versus
relative far-field noise.  Each point uses 120 realizations for each
component count.}
\label{fig:example1-accuracy}
\end{figure}

The fixed-frequency improvement agrees with Theorem~\ref{thm:variance_ratio}.
At the largest retained order, \(p=8\), the predicted standard-deviation ratio
between the endpoint row and the total-order estimator is
\(\sqrt{\binom{16}{8}}=113.446\).  Figure~\ref{fig:example1-noisy-spectrum}
shows the effect on one four-component realization at noise \(10^{-4}\).
The pair \((\sigma_4/\sigma_1,\sigma_5/\sigma_1)\) changes from
\((3.287\times10^{-2},2.000\times10^{-2})\) for the single row to
\((2.285\times10^{-2},2.98\times10^{-4})\) for all channels at \(k=2\) and to
\((2.157\times10^{-2},1.60\times10^{-4})\) for seven frequencies.  The signal
singular value is essentially preserved while the noise-induced tail is
reduced by roughly two orders of magnitude.

\begin{figure}[htbp]
\centering
\includegraphics[width=0.72\textwidth]{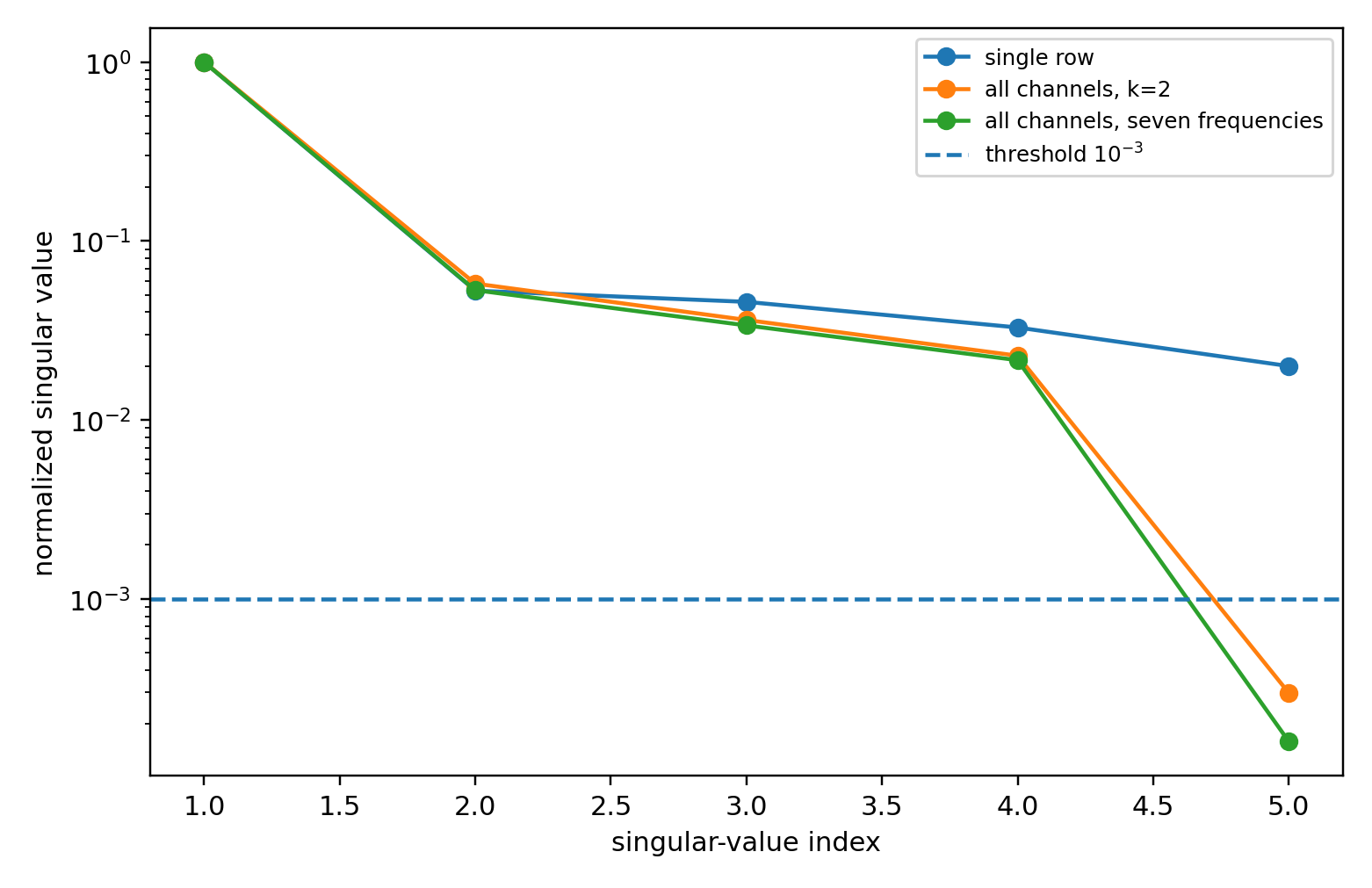}
\caption{Normalized Hankel spectrum for one \(N_c=4\) realization at relative
noise \(10^{-4}\).  Total-order averaging suppresses the fifth singular value
without changing the four-dimensional signal subspace.}
\label{fig:example1-noisy-spectrum}
\end{figure}

\subsection{Phase-center localization}
\label{subsec:example2}

We retain the four-disk geometry in \eqref{eq:example1-geometries} and supply
the correct count \(N_c=4\) to the shifted pencil, thereby separating node
recovery from rank selection.  The three estimates are the single row at
\(k=2\), the leading seven-frequency multichannel fit, and the second-order fit
of \eqref{eq:second_gls}.  For recovered centers \(\widehat z_j\), the error is
\begin{align}
E_{\rm loc}
=
\min_{\pi\in\mathfrak S_{N_c}}
\max_{1\le j\le N_c}
\|z_j-\widehat z_{\pi(j)}\|_2.
\label{eq:example2-location-error}
\end{align}
At each nonzero noise level we use 120 realizations and report the median.

The single-row clean error is \(1.71\times10^{-2}\) and grows to
\(1.80\times10^{-1}\) at noise \(10^{-4}\) and to \(1.02\) at \(10^{-3}\).
The leading multichannel curve is nearly flat: its median is
\(1.85\times10^{-2}\) in clean data, \(1.86\times10^{-2}\) at \(10^{-4}\), and
\(2.24\times10^{-2}\) at \(10^{-3}\).  After noise amplification has been
suppressed, finite-frequency bias becomes visible.  The second-order fit
reduces the clean error to \(2.41\times10^{-3}\) and gives
\(5.81\times10^{-3}\) at \(10^{-4}\), but its extra fitted parameter increases
variance at larger noise; at \(10^{-3}\) its median error is
\(4.54\times10^{-2}\).  This is the bias--variance trade-off described by
Proposition~\ref{prop:second_order}.

\begin{figure}[htbp]
\centering
\includegraphics[width=0.74\textwidth]{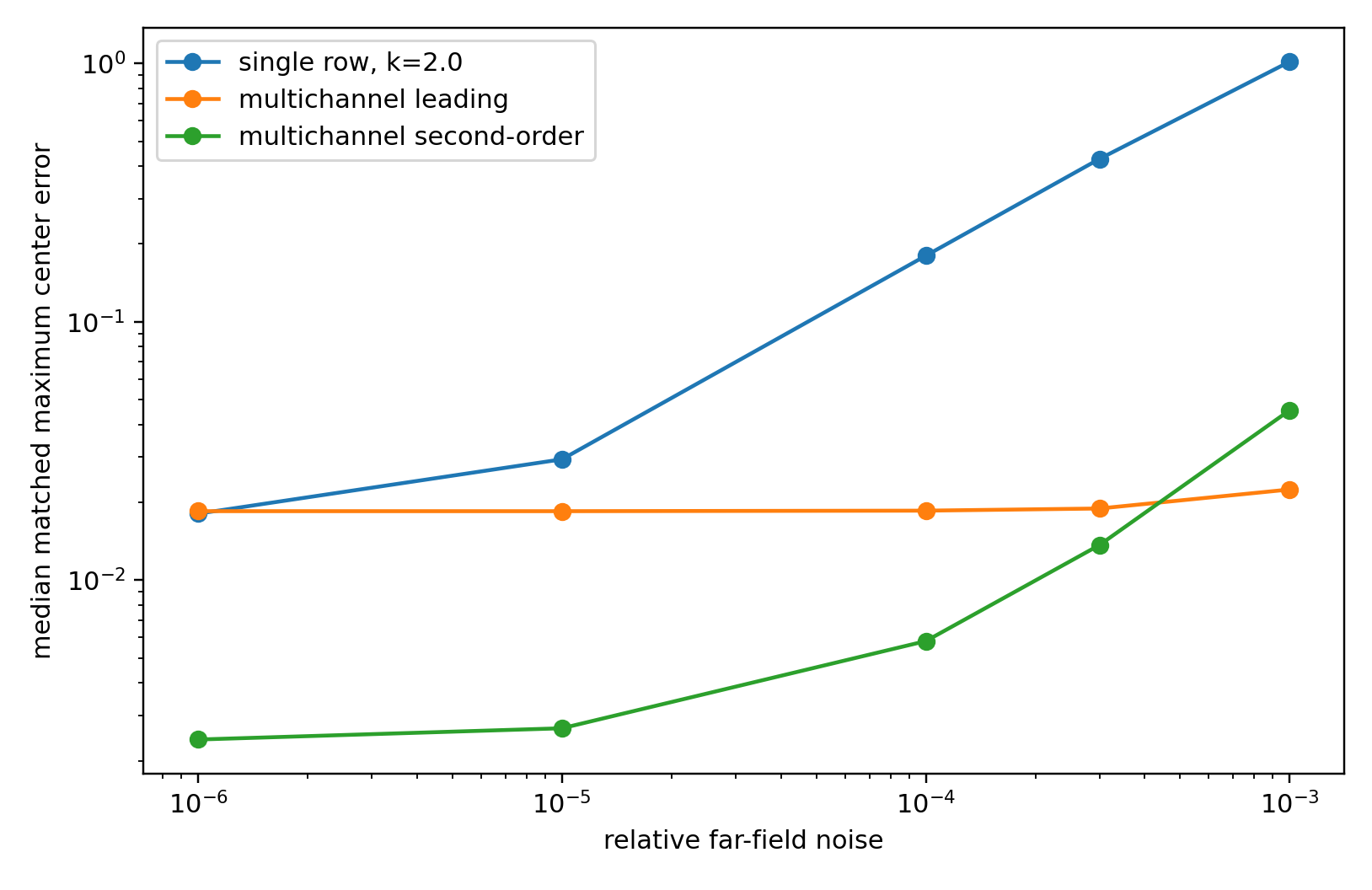}
\caption{Median matched maximum phase-center error.  The correct count
\(N_c=4\) is supplied to the pencil.}
\label{fig:example2-localization}
\end{figure}

Representative reconstructions are shown in Figure~\ref{fig:example2-recovery}.
At noise \(10^{-4}\) the single-row nodes are visibly displaced, whereas both
multichannel reconstructions remain close to the true centers.

\begin{figure}[htbp]
\centering
\begin{subfigure}[t]{0.48\textwidth}
\centering
\includegraphics[width=\textwidth]{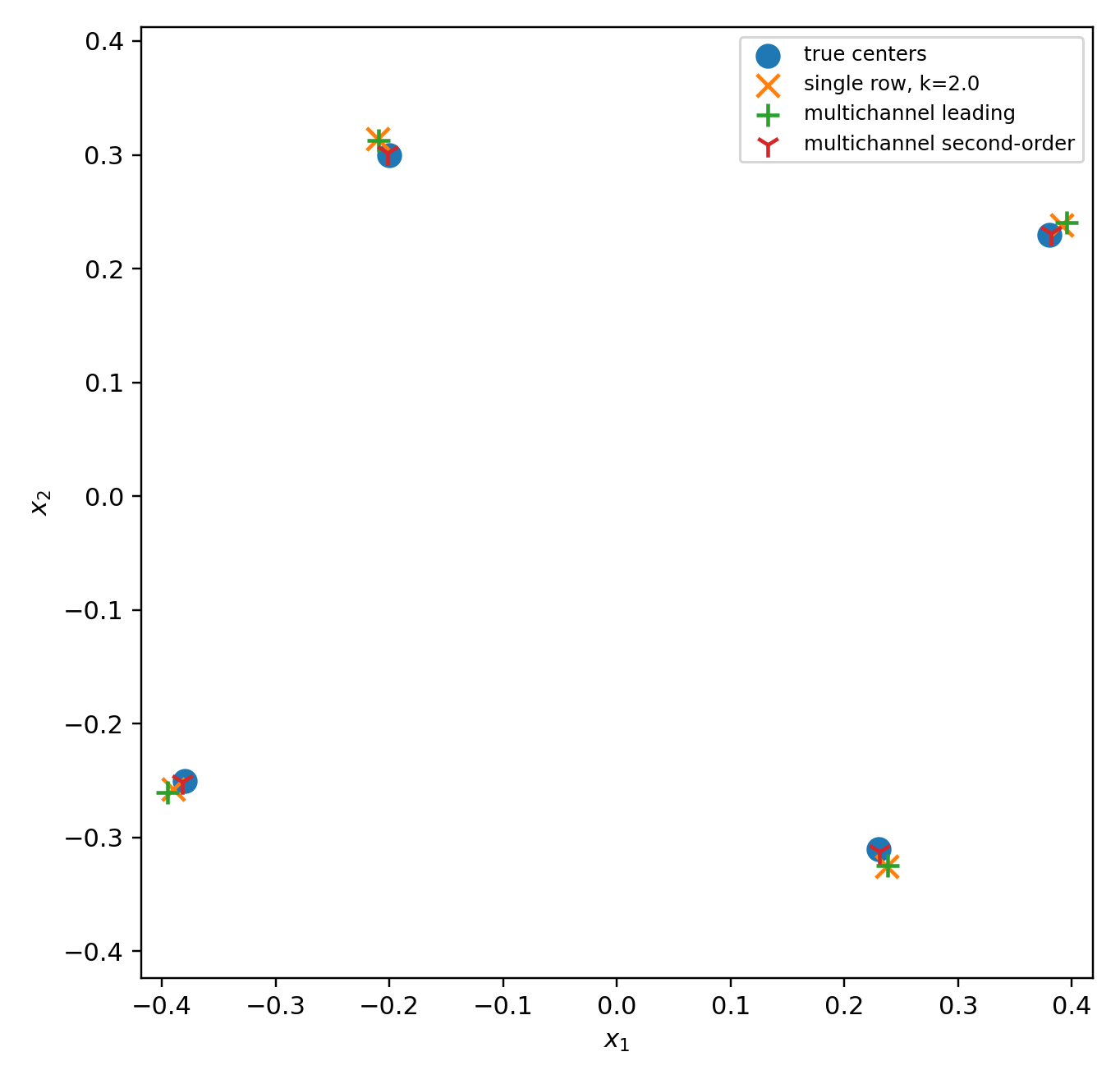}
\caption{No added noise.}
\end{subfigure}
\hfill
\begin{subfigure}[t]{0.48\textwidth}
\centering
\includegraphics[width=\textwidth]{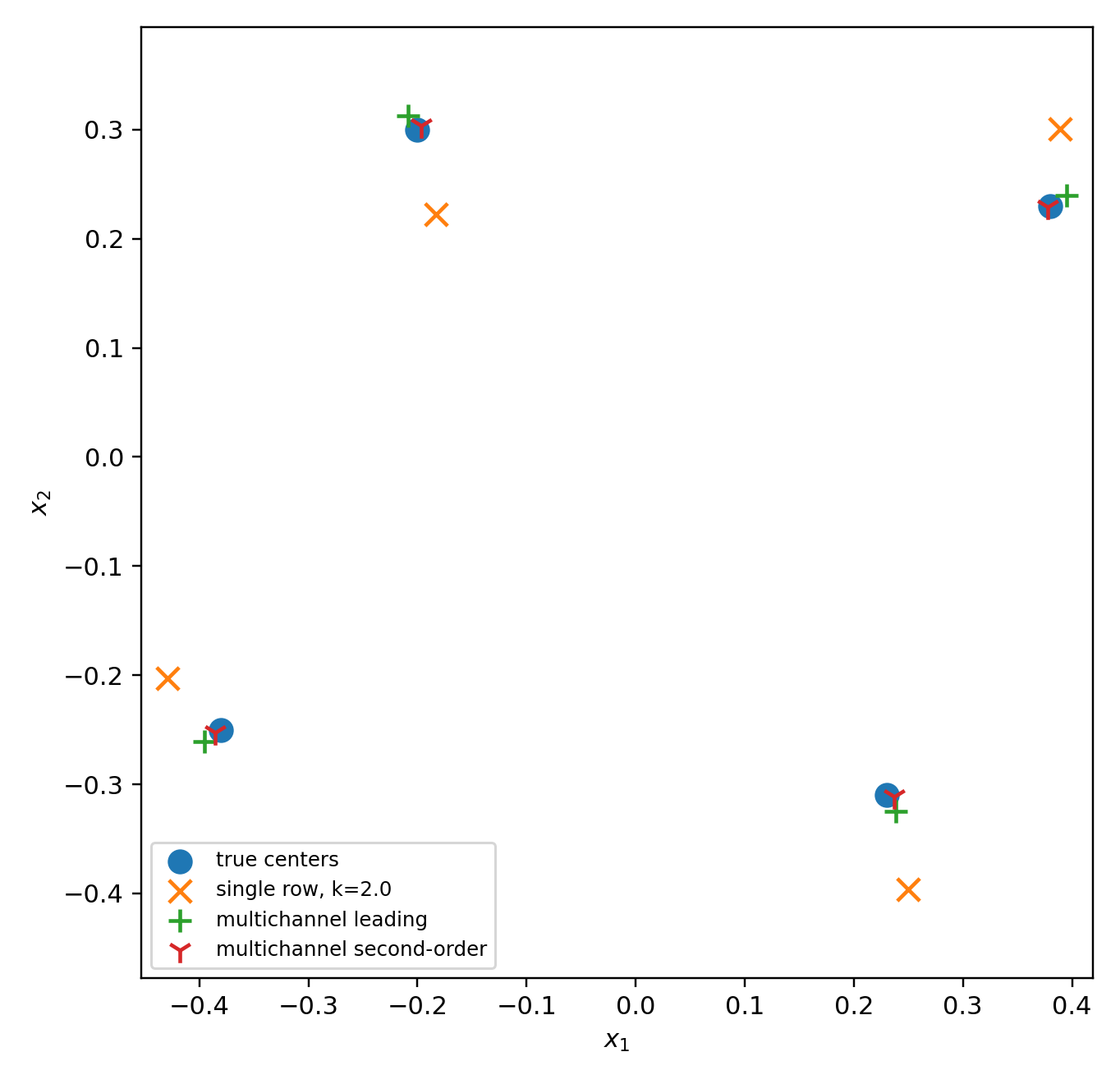}
\caption{Relative noise \(10^{-4}\).}
\end{subfigure}
\caption{Representative phase-center reconstructions for the single-row,
leading multichannel, and second-order multichannel pencils.}
\label{fig:example2-recovery}
\end{figure}

\subsection{Off-center cavity detection}
\label{subsec:example3}

Consider a unit-contrast disk of radius \(0.20\) centered at the origin with a
circular cavity of radius \(0.06\) centered at \((c,0)\).  For a \(5\times5\)
signed Hankel matrix define
\begin{align}
T(c)=\frac{\sigma_2(H_J^{\rm sgn}(c))}{\sigma_1(H_J^{\rm sgn}(c))}.
\label{eq:example3-statistic}
\end{align}
The signed phase-center model has one effective node at \(c=0\) and two when the
material and cavity centers are distinct.

Figure~\ref{fig:example3-clean} shows that the clean statistic is essentially
unchanged by multichannel averaging.  At \(c=0.05\) the row, fixed-frequency
multichannel, and seven-frequency values are
\(2.8231\times10^{-4}\), \(2.8277\times10^{-4}\), and
\(2.8058\times10^{-4}\), respectively; at \(c=0.10\) all three are approximately
\(1.12\times10^{-3}\).  The additional channels therefore lower the noise floor
rather than create a larger clean cavity signal.

\begin{figure}[htbp]
\centering
\includegraphics[width=0.73\textwidth]{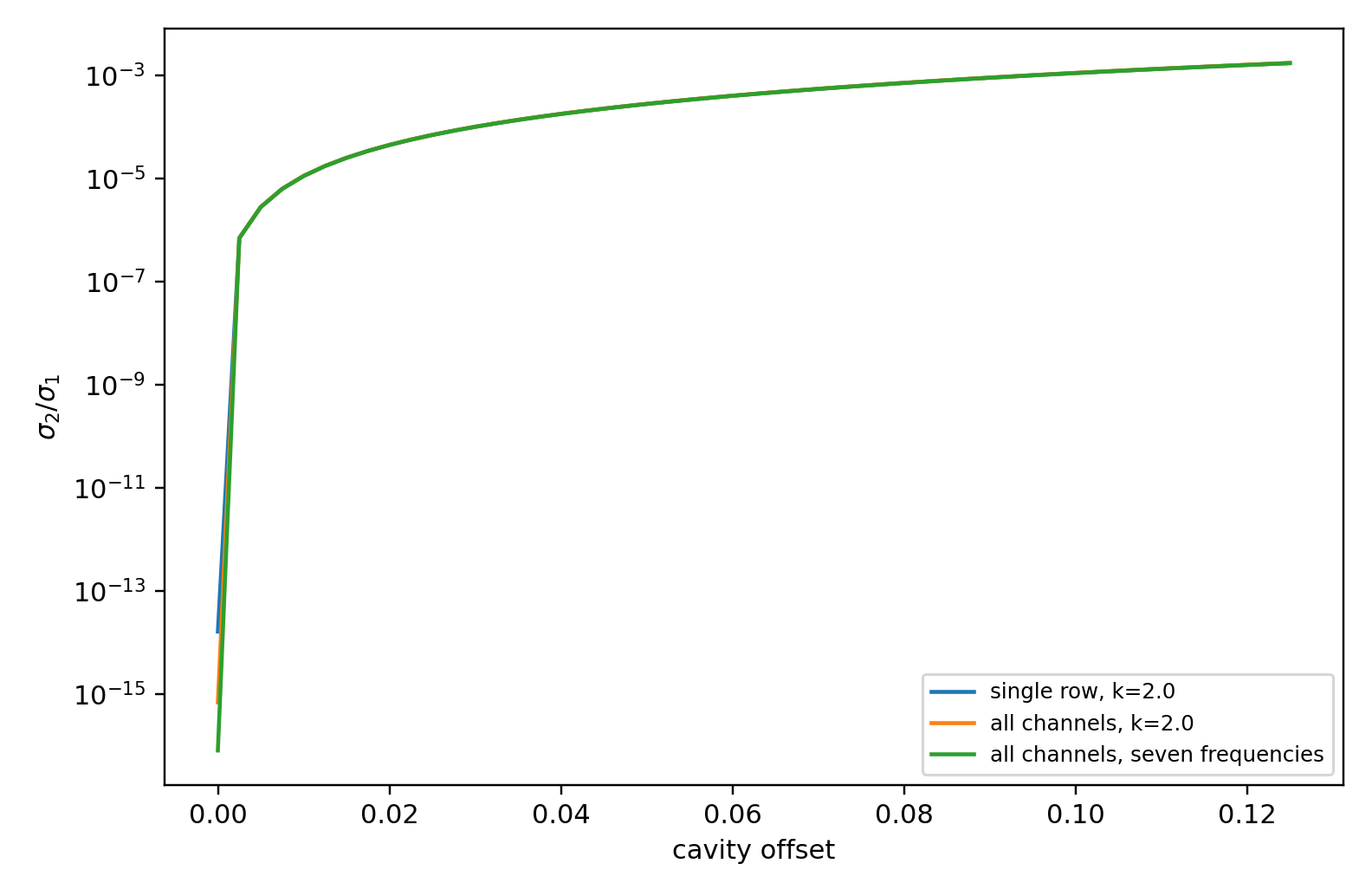}
\caption{Clean signed statistic \(T(c)=\sigma_2/\sigma_1\) versus cavity
offset.  The three estimators give nearly identical clean signals and collapse
to machine precision at \(c=0\).}
\label{fig:example3-clean}
\end{figure}

For noisy data, the concentric case \(c=0\) is used only to calibrate a
method-specific \(99\%\) null quantile.  Detection is declared when
\eqref{eq:example3-statistic} exceeds that threshold.  At the alternative
\(c=0.05\), the single-row detection probabilities stay at the null scale even
for very small noise.  The fixed-frequency multichannel detector has unit
detection through \(3\times10^{-5}\) and then loses separation by \(10^{-4}\).
The seven-frequency detector remains perfect through \(10^{-4}\); at
\(3\times10^{-4}\) and \(10^{-3}\) it returns to the null level.  This behavior
matches Corollary~\ref{cor:signed_variance}: averaging improves detectability
of an existing second signed node but cannot remove the coincident-node
degeneracy.

\begin{figure}[htbp]
\centering
\includegraphics[width=0.73\textwidth]{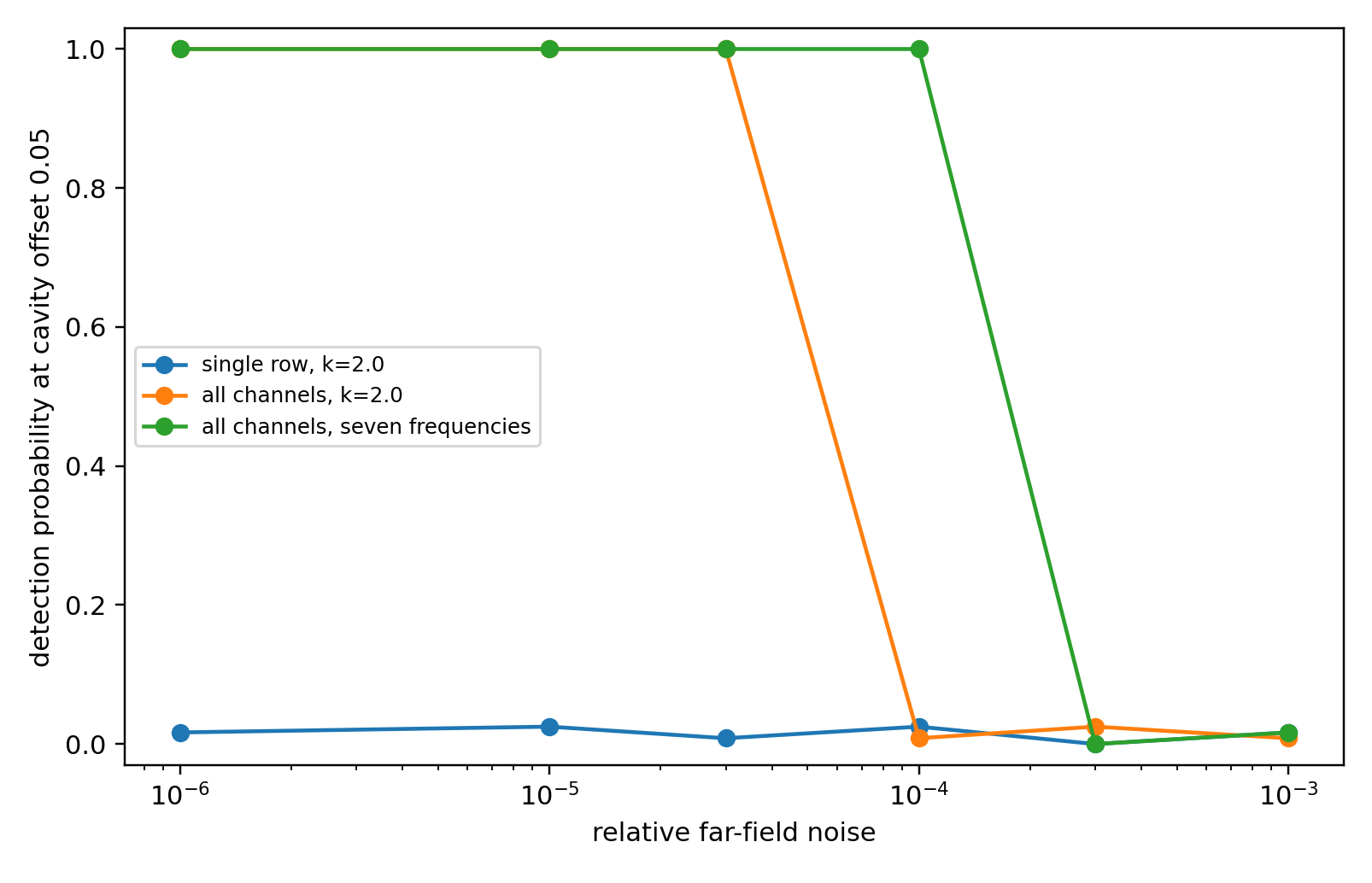}
\caption{Null-calibrated detection probability for an off-center cavity with
\(c=0.05\).  The \(99\%\) null threshold is calibrated separately for each
estimator and noise level.}
\label{fig:example3-detection}
\end{figure}

\subsection{Concentric cavities and radial recovery}
\label{subsec:example4}

We now set \(c=0\) and compare a filled disk of radius \(R=0.20\) with a
concentric annulus having the same outer radius and inner radius \(\rho=0.06\).
Proposition~\ref{prop:zero_total_radial} predicts that all coefficients with
nonzero total order vanish for a centered radial contrast.  Figure~\ref{fig:example4-phase-radial}(a)
confirms the resulting phase-center degeneracy: the normalized spectra are
\((1,0,0)\) for the filled disk and
\((1,1.08\times10^{-17},1.08\times10^{-17})\) for the annulus.  The largest
positive-total-order coefficient for \(p=1,\ldots,5\) is of order \(10^{-18}\)
relative to the zero-order coefficient.

For the radial construction we symmetrize \(a_{m,-m}\) and \(a_{-m,m}\) and use
five frequencies ranging from \(1.2\) to \(2.4\), with reference radius
\(R_*=0.25\).  The normalized radial spectra in
Figure~\ref{fig:example4-phase-radial}(b) are
\begin{align}
\text{filled disk: }&(1,3.38\times10^{-5},7.61\times10^{-7}),\nonumber\\
\text{annulus: }&(1,1.98\times10^{-2},1.06\times10^{-6}).
\label{eq:example4-radial-spectra}
\end{align}
Thus the filled disk is numerically rank one, while the annulus has a clear
second radial singular value.  The radial pencil of
Theorem~\ref{thm:radial_rank} returns
\begin{align}
\widehat\rho=0.060078,
\qquad
\widehat R=0.200006,
\label{eq:example4-radii}
\end{align}
with absolute errors \(7.8\times10^{-5}\) and \(5.5\times10^{-6}\).

\begin{figure}[htbp]
\centering
\begin{subfigure}[t]{0.48\textwidth}
\centering
\includegraphics[width=\textwidth]{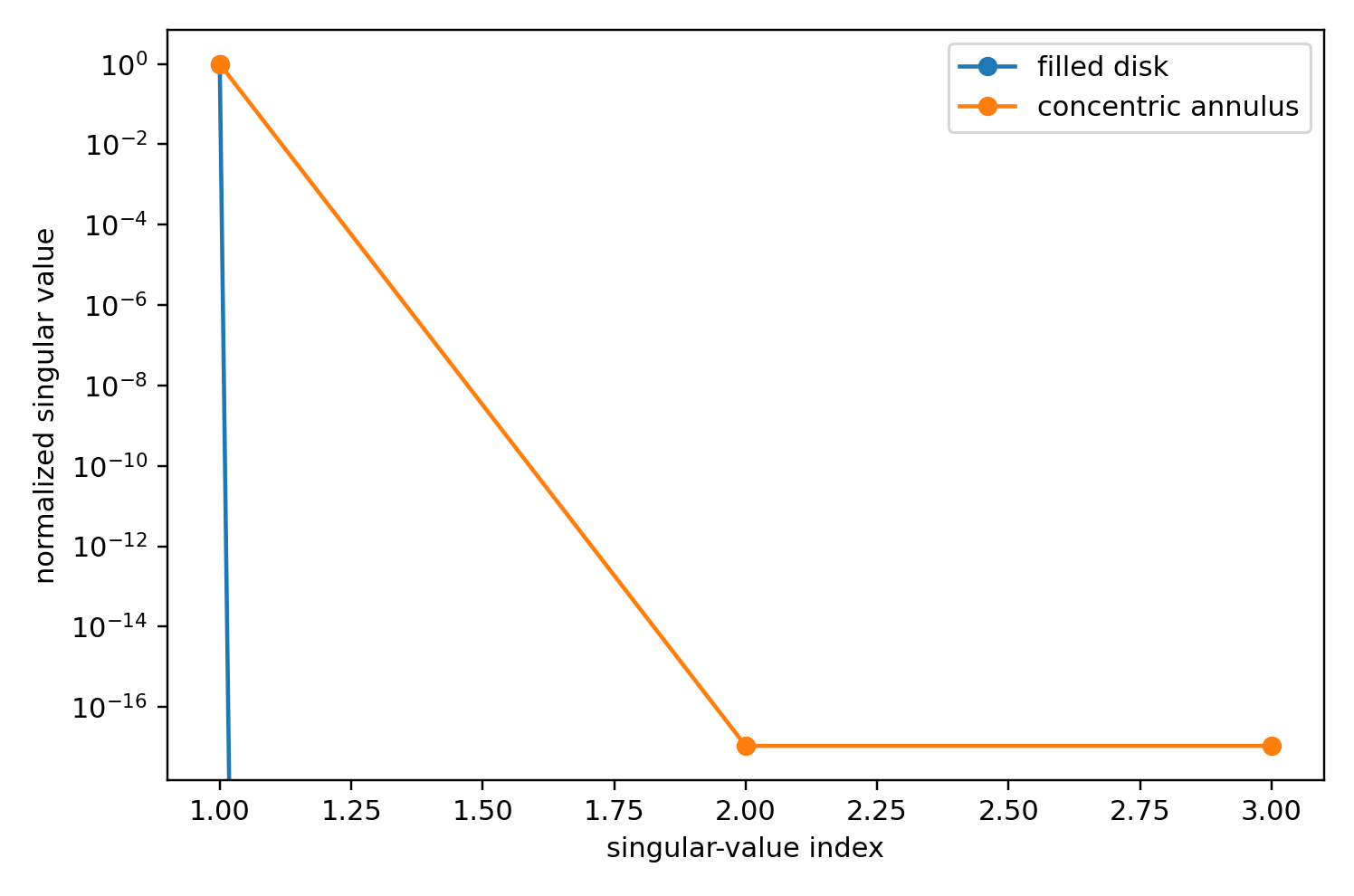}
\caption{Phase-center spectra.}
\end{subfigure}
\hfill
\begin{subfigure}[t]{0.48\textwidth}
\centering
\includegraphics[width=\textwidth]{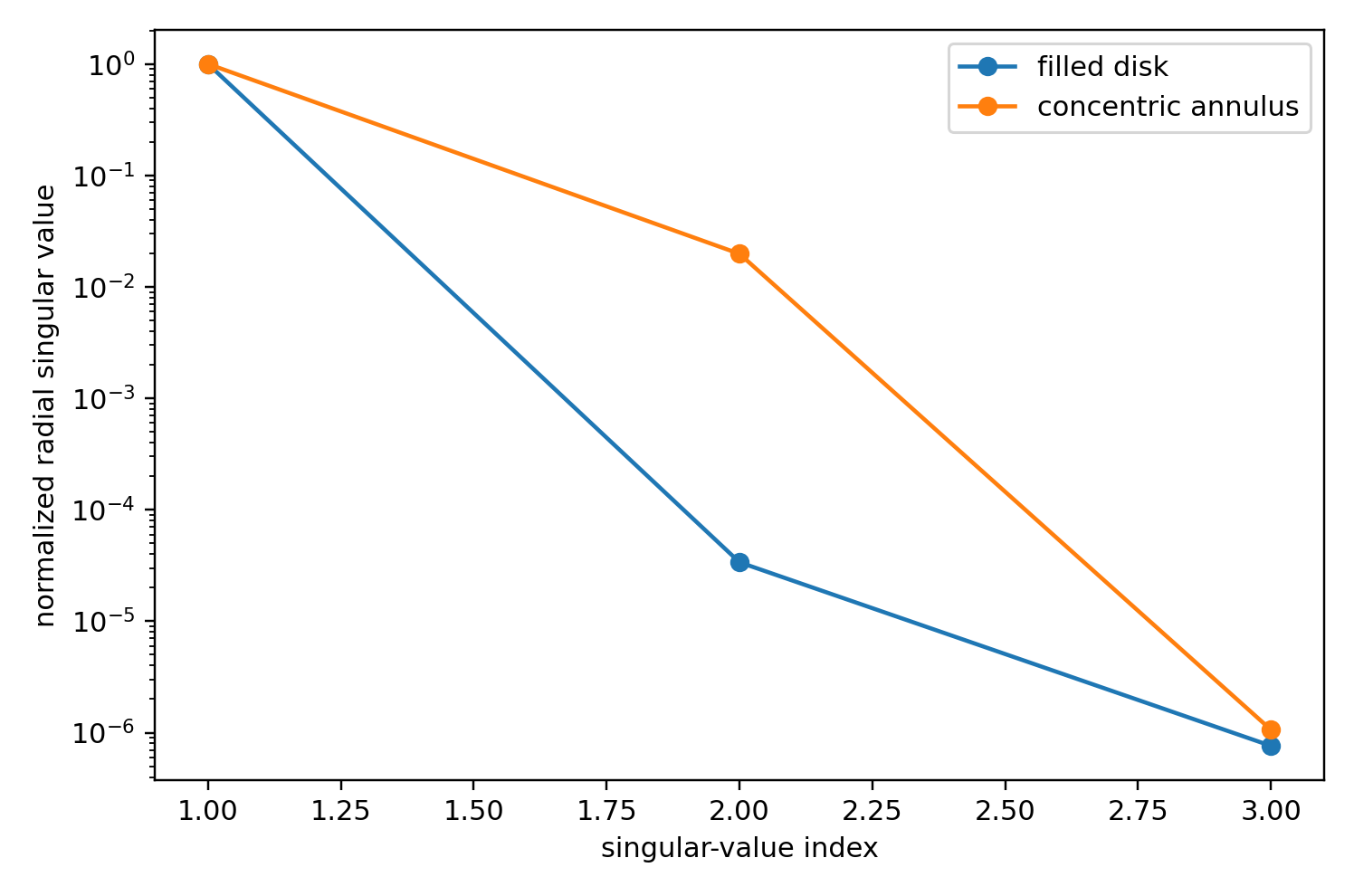}
\caption{Radial spectra.}
\end{subfigure}
\caption{Filled disk versus concentric annulus.  Positive total order cannot
separate the concentric cavity, whereas zero total order produces a second
radial node.}
\label{fig:example4-phase-radial}
\end{figure}

The finite-frequency radial moments are estimated by the weighted intercept
fit of Section~\ref{subsec:radial_stability}.  To check its deterministic
order, all five frequencies are multiplied by a common scale \(s\) and the
radial moment of order \(m=2\) is reconstructed.  The measured log--log slope
is \(3.99532\), in agreement with the \(O(s^4)\) intercept error implied by
\eqref{eq:radial_regression_remainder}--\eqref{eq:radial_bias_bound}.

\begin{figure}[htbp]
\centering
\includegraphics[width=0.69\textwidth]{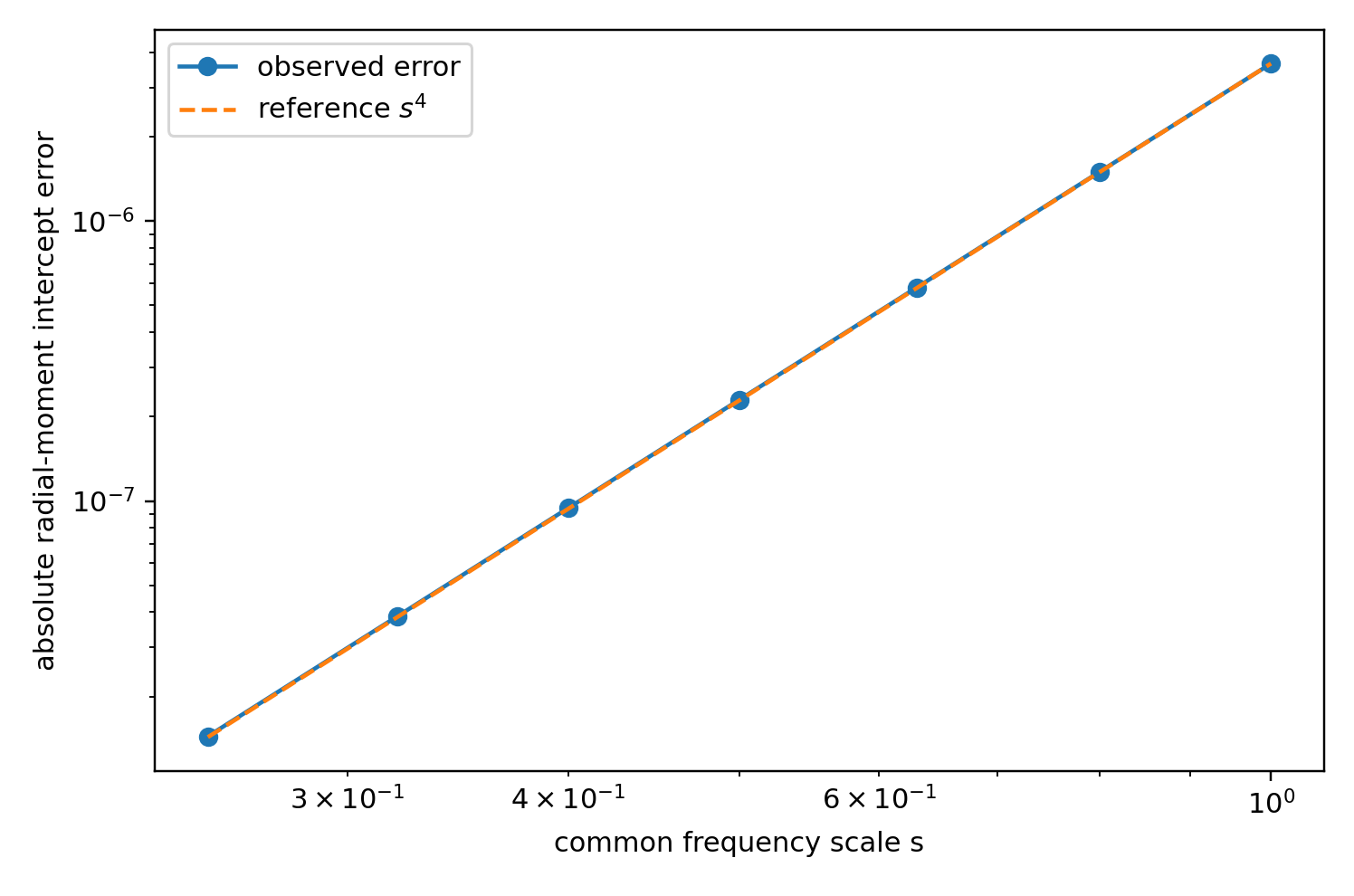}
\caption{Absolute radial-moment intercept error under a common frequency
scaling.  The observed slope \(3.99532\) agrees with the predicted fourth-order
remainder after the quadratic term is fitted.}
\label{fig:example4-bias}
\end{figure}

Finally, define the radial detection statistic
\(T_{\rm rad}=\sigma_2(H_J^{\rm rad})/\sigma_1(H_J^{\rm rad})\).
At each noise level a \(99\%\) null threshold is calibrated from 200 noisy
filled-disk realizations and tested on 120 independent annulus realizations.
Figure~\ref{fig:example4-detection} shows unit detection through relative
noise \(10^{-4}\).  At \(3\times10^{-4}\) the threshold rises to
\(2.40\times10^{-2}\) while the median annulus statistic is
\(1.92\times10^{-2}\), and the detection probability falls to \(0.20\); at
\(10^{-3}\) it is \(0.017\).  The transition is consistent with the gap condition
in Corollary~\ref{cor:radial_rank_stability}.

\begin{figure}[htbp]
\centering
\includegraphics[width=0.70\textwidth]{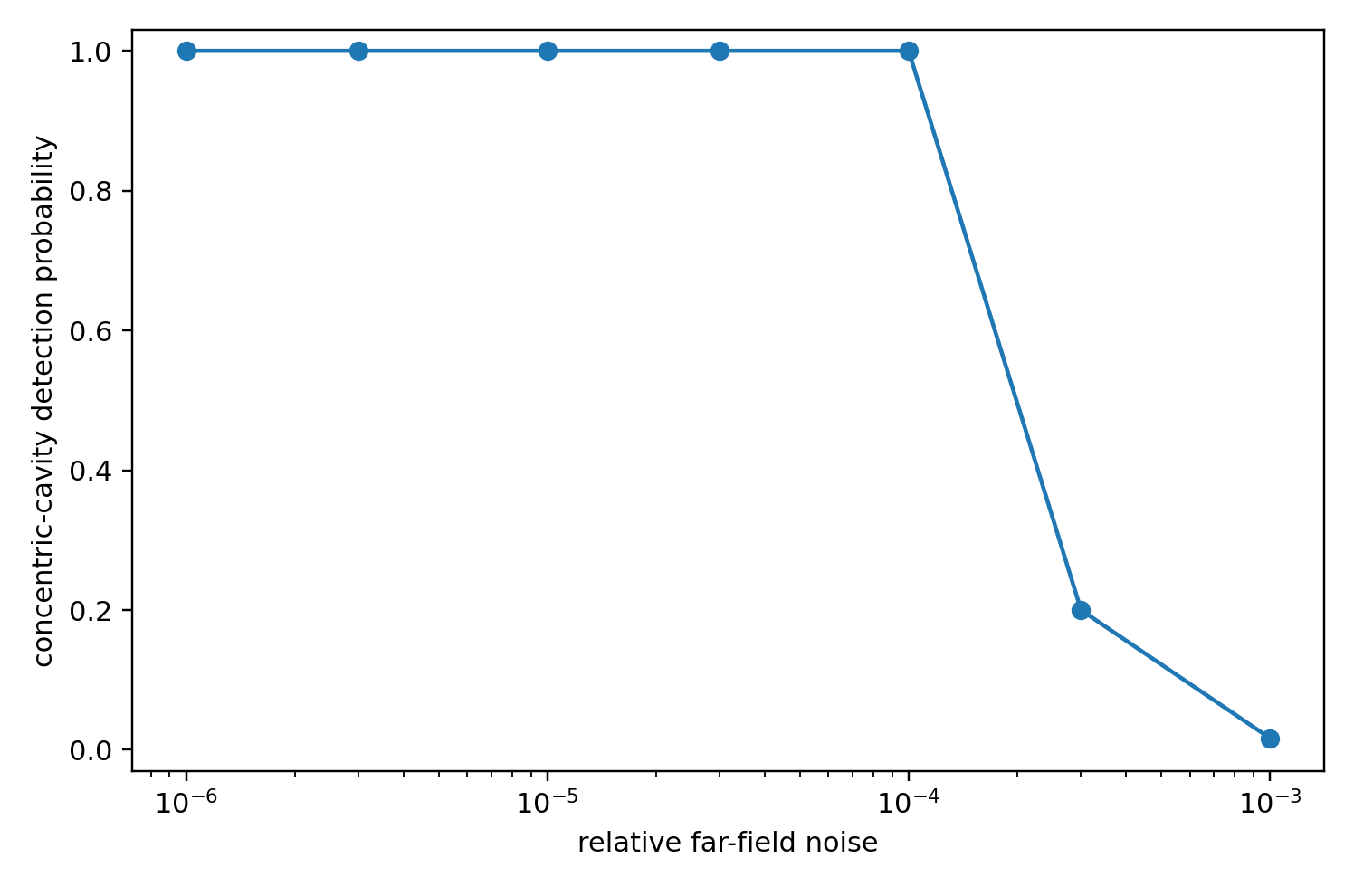}
\caption{Null-calibrated detection probability for the concentric annulus
using the radial Fourier--Hankel statistic.}
\label{fig:example4-detection}
\end{figure}

\subsection{Multiple concentric interfaces}
\label{subsec:example5}

To test the higher-rank radial result, consider
\begin{align}
q_0(r)
=
\mathbf 1_{[0,0.40)}(r)
-
\mathbf 1_{[0,0.25)}(r)
+
\mathbf 1_{[0,0.10)}(r),
\label{eq:example5-profile}
\end{align}
which has interfaces at \(R_1=0.10\), \(R_2=0.25\), and \(R_3=0.40\).
Using seven frequencies from \(0.8\) to \(2.0\) and \(R_*=0.45\), the normalized
singular values of the reconstructed \(4\times4\) radial Hankel matrix are
\begin{align}
1,
\qquad 5.74\times10^{-2},
\qquad 1.09\times10^{-3},
\qquad 3.80\times10^{-6}.
\label{eq:example5-spectrum}
\end{align}
The first three singular values form the signal subspace predicted by
Theorem~\ref{thm:radial_rank}.  The shifted pencil returns
\begin{align}
\widehat R_1=0.099199,
\qquad
\widehat R_2=0.250675,
\qquad
\widehat R_3=0.400021,
\label{eq:example5-recovered-radii}
\end{align}
so the largest absolute radius error is \(8.01\times10^{-4}\).

\begin{figure}[htbp]
\centering
\includegraphics[width=0.68\textwidth]{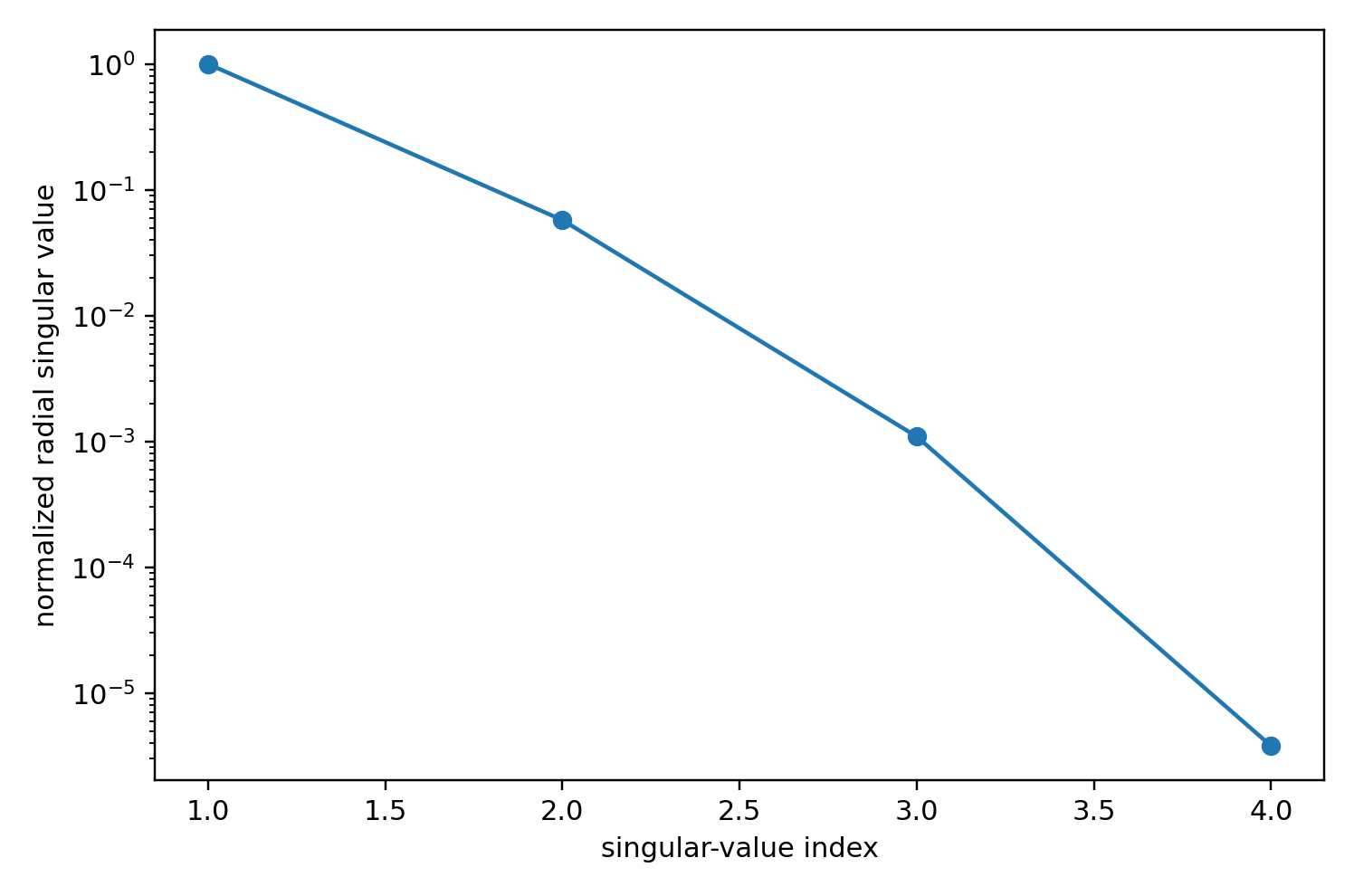}
\caption{Normalized radial Hankel spectrum for the three-interface profile
\eqref{eq:example5-profile}.}
\label{fig:example5-spectrum}
\end{figure}

\subsection{Intrinsic Hankel conditioning}
\label{subsec:example6}

The preceding procedures reduce stochastic and finite-frequency errors but
cannot change the exact Hankel gap.  To isolate this effect, we generate 1000
exact four-node phase-center sequences with centers in \([-0.45,0.45]^2\) and
minimum separation \(0.22\).  No measurement noise or Bessel approximation is
used.  For each configuration we record
\begin{align}
G=\frac{\sigma_4(H_J)}{\sigma_1(H_J)}.
\label{eq:example6-gap}
\end{align}
The empirical \(10\%\) quantile, median, and \(90\%\) quantile are
\(2.61\times10^{-4}\), \(1.02\times10^{-3}\), and \(4.34\times10^{-3}\),
respectively.  Only \(51.2\%\) of the exact rank-four matrices lie above the
fixed reference threshold \(10^{-3}\).

\begin{figure}[htbp]
\centering
\includegraphics[width=0.70\textwidth]{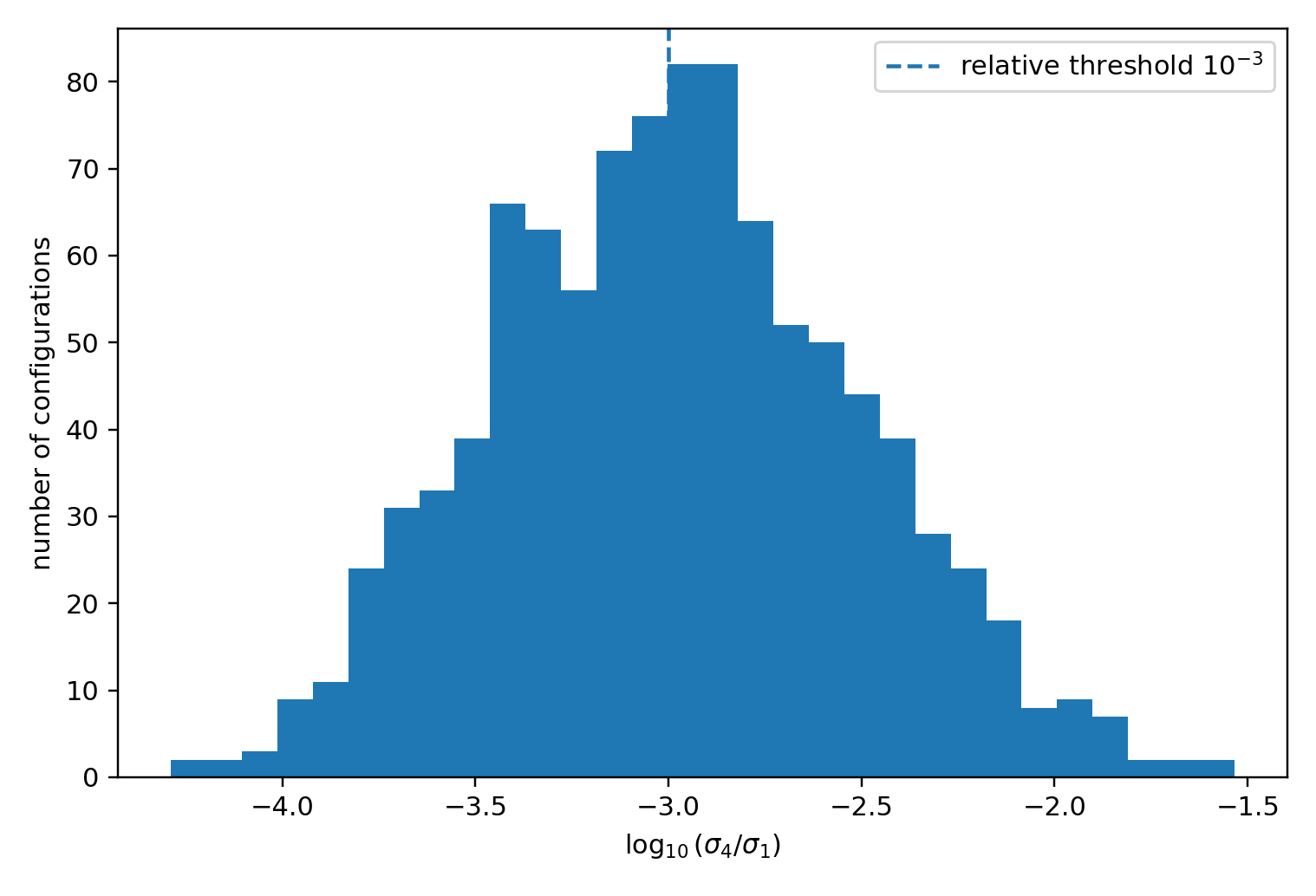}
\caption{Distribution of the exact normalized fourth singular value over
1000 four-node configurations.  The dashed line marks \(10^{-3}\).}
\label{fig:example6-gap}
\end{figure}

This experiment is the numerical counterpart of
Theorem~\ref{thm:rank_stability}.  Multichannel estimation can reduce the
perturbation applied to the Hankel matrix, but the admissible threshold still
depends on the smallest nonzero singular value of the exact configuration.

\subsection{Full-wave Helmholtz data beyond the Born approximation}
\label{subsec:example7}

The preceding experiments isolate the algebraic and stochastic mechanisms
using Born data.  We now keep the reconstruction unchanged but generate the far
field from the full Helmholtz model \eqref{eq:helmholtz}.  This experiment is a
model-mismatch test: no result above asserts an exact finite-rank
Fourier--Hankel structure for nonlinear multiple scattering.

We use the four centers in \eqref{eq:example1-geometries}, increase the disk
radius to \(0.05\), and set
\begin{align}
q_\alpha(x)
&=
\alpha\sum_{j=1}^{4}\mathbf 1_{D_j}(x),
\qquad
\alpha>0.
\label{eq:example7-contrast}
\end{align}
The frequency is fixed at \(k=2\).  For each incident direction, the total
field is computed from the Lippmann--Schwinger equation equivalent to
\eqref{eq:helmholtz},
\begin{align}
u(x,d;k)
&=
u^i(x,d;k)
+k^2\int_D \Phi_k(x-y)q_\alpha(y)u(y,d;k)\,\dd y,
\label{eq:example7-ls}
\end{align}
where \(\Phi_k=(\ii/4)H_0^{(1)}(k|\cdot|)\).  The volume integral is discretized
on cells of width \(h=0.006\) inside the four disks, giving 860 unknowns.  The
logarithmic self interaction is integrated by an equal-area disk correction.
The full-wave far field is then evaluated from \eqref{eq:farfield_exact}.  A
refinement from \(h=0.008\) to \(h=0.006\) changes the reported full-wave/Born
discrepancy by less than \(0.4\%\) and the fifth normalized singular value by
less than \(1.7\%\) for the representative contrasts \(\alpha=1,8,16,20\).

To quantify the departure from the model used by the reconstruction, define
\begin{align}
E_{\rm Born}(\alpha)
&=
\frac{\|U_\alpha^\infty-U_{\alpha,{\rm Born}}^\infty\|_F}
{\|U_\alpha^\infty\|_F}.
\label{eq:example7-born-discrepancy}
\end{align}
The same fixed-frequency total-order estimator as in
\eqref{eq:all_estimator} is applied directly to the full-wave data, with no
nonlinear correction.  A \(5\times5\) Hankel matrix is formed from moments
through order eight.  Table~\ref{tab:example7-fullwave} reports the fourth and
fifth normalized singular values and the numerical rank from
\eqref{eq:numerical-rank}.  The localization error is computed with the true
count \(N_c=4\) supplied to the pencil, so that node bias is separated from
rank-selection failure.

\begin{table}[htbp]
\centering
\caption{Born-derived Fourier--Hankel reconstruction applied to full-wave
Helmholtz data.  The full-wave/Born discrepancy is
\eqref{eq:example7-born-discrepancy}; localization uses the true count
\(N_c=4\).}
\label{tab:example7-fullwave}
\begin{tabular}{ccccc}
\toprule
\(\alpha\)
& \(E_{\rm Born}\)
& \(\sigma_4/\sigma_1\)
& \(\sigma_5/\sigma_1\)
& \(\operatorname{rank}_{10^{-3}}\; /\; E_{\rm loc}\)\\
\midrule
0.25 & 0.0056 & \(2.31\times10^{-2}\) & \(2.24\times10^{-4}\) & \(4\;/\;0.0262\)\\
1    & 0.0223 & \(2.33\times10^{-2}\) & \(2.37\times10^{-4}\) & \(4\;/\;0.0272\)\\
4    & 0.0891 & \(2.40\times10^{-2}\) & \(3.15\times10^{-4}\) & \(4\;/\;0.0327\)\\
8    & 0.1777 & \(2.52\times10^{-2}\) & \(4.78\times10^{-4}\) & \(4\;/\;0.0431\)\\
12   & 0.2659 & \(2.69\times10^{-2}\) & \(7.01\times10^{-4}\) & \(4\;/\;0.0564\)\\
16   & 0.3531 & \(2.91\times10^{-2}\) & \(9.95\times10^{-4}\) & \(4\;/\;0.0716\)\\
20   & 0.4391 & \(3.18\times10^{-2}\) & \(1.38\times10^{-3}\) & \(5\;/\;0.0884\)\\
\bottomrule
\end{tabular}
\end{table}

\begin{figure}[htbp]
\centering
\includegraphics[width=0.72\textwidth]{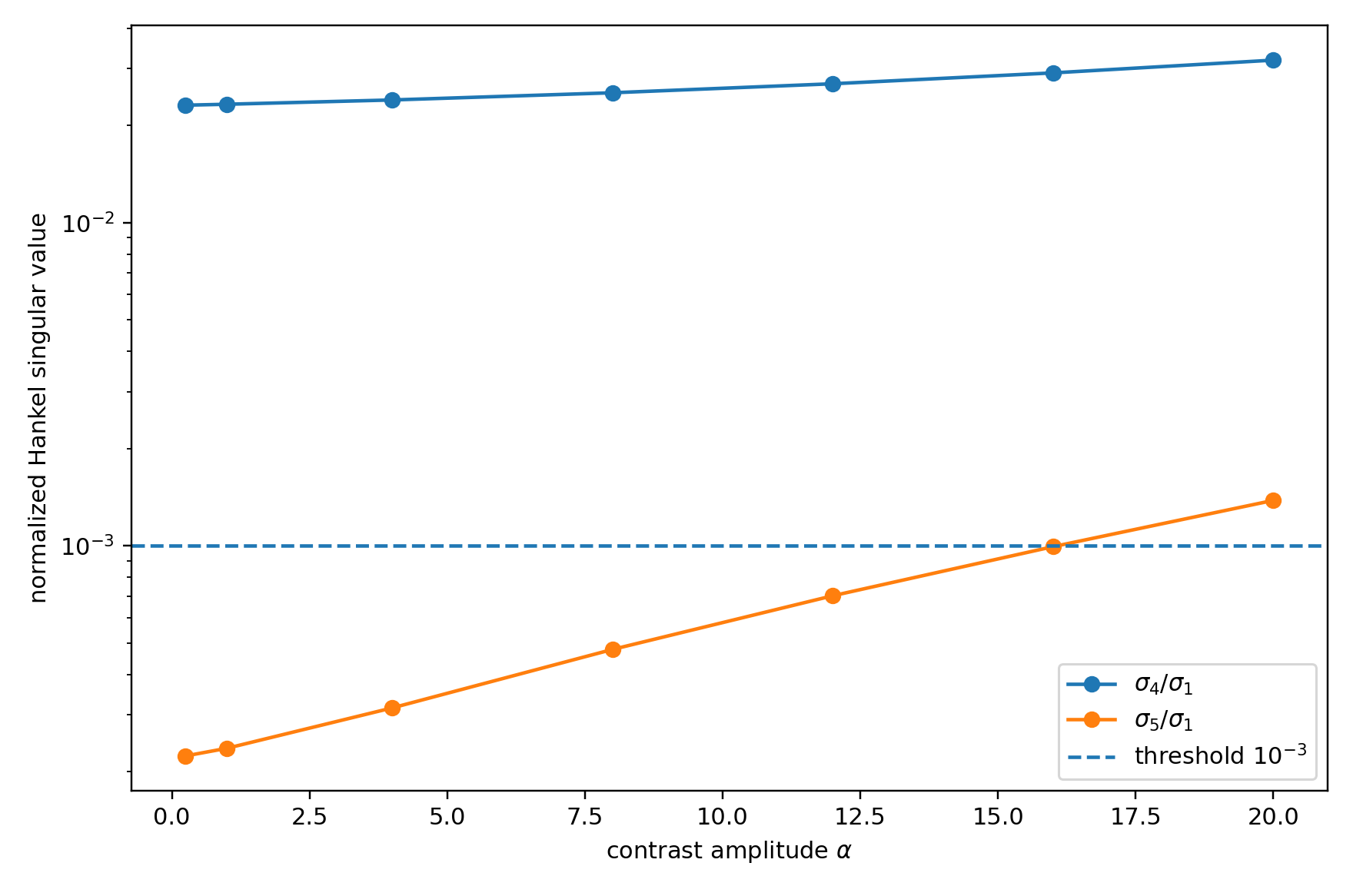}
\caption{Fourth and fifth normalized Hankel singular values obtained by
applying the Born-derived total-order reduction to full-wave Helmholtz data.
The fifth singular value tracks the growing Hankel tail associated with model
mismatch.  It remains below the reference threshold through \(\alpha=16\)
and crosses it at \(\alpha=20\).}
\label{fig:example7-fullwave}
\end{figure}

The results distinguish the exact Born theory from the observed robustness of
the reconstruction.  At \(\alpha=8\), the full-wave data differ from their Born
counterparts by about \(18\%\), yet the four-dimensional signal subspace
remains clearly separated from the fifth singular value.  Even at
\(\alpha=16\), where the discrepancy is about \(35\%\), the fixed threshold
still returns four components, although the localization bias has increased.
At \(\alpha=20\), the fifth singular value crosses \(10^{-3}\) and the same
rank rule returns five components.  Thus the Born-derived reduction remains
numerically informative well beyond an infinitesimal model mismatch, but the
nonlinear scattering error eventually appears as an additional Hankel tail
and cannot be removed by multichannel averaging.  This is consistent with the scope of the
analysis: the extra channels reduce measurement noise, whereas full-wave model
error is a separate deterministic perturbation.

\section{Conclusion}
\label{sec:conclusion}

The two-angle Fourier matrix supports two distinct finite-moment reductions.
Positive total order carries phase-center information.  Replacing a single
row coefficient by the nonnegative total-order channel ensemble reduces the
variance of the recovered high-order moments and markedly improves counting,
localization, and signed off-center cavity detection.  Zero total order
retains radial Bessel information.  For a finitely layered angular average,
multifrequency extrapolation yields a Hankel sequence in the squared interface
radii, allowing concentric interfaces to be counted and localized even when
the phase-center rank is blind to the cavity.

The analysis also identifies what the stabilization cannot change.  The
recoverable rank is controlled by the smallest nonzero singular value of the
exact Hankel matrix, and the node error is amplified by Vandermonde
conditioning.  Additional channels reduce the perturbation of the moments;
they do not create a geometry-independent singular gap.  Likewise, the radial
construction resolves interfaces of a finitely layered angular average; it does
not recover the shape of an arbitrary same-centered cavity.
The full-wave experiment shows that the Born-derived phase-center rank can
remain useful under appreciable model mismatch, but it also shows where the
nonlinear scattering residual becomes large enough to generate a spurious
Hankel direction.  A corresponding full-wave error analysis, and its extension
to the radial reduction, remain open questions.
\bibliographystyle{plain}
\bibliography{myreference_final}

@book{CakoniColton2014,
	address = {New York},
	author = {F. Cakoni and D. Colton},
	doi = {10.1007/978-1-4614-8827-9},
	publisher = {Springer},
	series = {Applied Mathematical Sciences},
	title = {A Qualitative Approach to Inverse Scattering Theory},
	volume = {188},
	year = {2014}}

@book{ColtonKress2019,
	author = {D. Colton and R. Kress},
	edition = {4},
	publisher = {Springer},
	title = {Inverse Acoustic and Electromagnetic Scattering Theory},
	year = {2019}}

@book{Kirsch2011,
	address = {New York},
	author = {A. Kirsch},
	doi = {10.1007/978-1-4419-8474-6},
	edition = {2},
	publisher = {Springer},
	series = {Applied Mathematical Sciences},
	title = {An Introduction to the Mathematical Theory of Inverse Problems},
	volume = {120},
	year = {2011}}

@book{KirschGrinberg2008,
	address = {Oxford},
	author = {A. Kirsch and N. Grinberg},
	publisher = {Oxford University Press},
	title = {The Factorization Method for Inverse Problems},
	year = {2008}}

@article{Potthast2006,
	author = {R. Potthast},
	doi = {10.1088/0266-5611/22/2/R01},
	journal = {Inverse Problems},
	number = {2},
	pages = {R1--R47},
	title = {A survey on sampling and probe methods for inverse problems},
	volume = {22},
	year = {2006}}

@article{Ito2012,
	author = {K. Ito and B. Jin and J. Zou},
	doi = {10.1088/0266-5611/28/2/025003},
	journal = {Inverse Problems},
	number = {2},
	pages = {025003},
	title = {A direct sampling method to an inverse medium scattering problem},
	volume = {28},
	year = {2012}}

@article{Li2013,
	author = {J. Li and J. Zou},
	doi = {10.3934/ipi.2013.7.757},
	journal = {Inverse Problems and Imaging},
	number = {3},
	pages = {757--775},
	title = {A direct sampling method for inverse scattering using far-field data},
	volume = {7},
	year = {2013}}

@article{Liu2018,
	author = {J. Liu and J. Sun},
	doi = {10.1088/1361-6420/aaca90},
	journal = {Inverse Problems},
	number = {8},
	pages = {085007},
	title = {Extended sampling method in inverse scattering},
	volume = {34},
	year = {2018}}

@article{AmmariIakovlevaMoskow2003,
	author = {H. Ammari and E. Iakovleva and S. Moskow},
	doi = {10.1137/S0036141001392785},
	journal = {SIAM Journal on Mathematical Analysis},
	number = {4},
	pages = {882--900},
	title = {Recovery of Small Inhomogeneities from the Scattering Amplitude at a Fixed Frequency},
	volume = {34},
	year = {2003}}

@book{AmmariKang2004,
	address = {Berlin},
	author = {H. Ammari and H. Kang},
	doi = {10.1007/b98245},
	publisher = {Springer},
	series = {Lecture Notes in Mathematics},
	title = {Reconstruction of Small Inhomogeneities from Boundary Measurements},
	volume = {1846},
	year = {2004}}

@book{AmmariKang2007,
	address = {New York},
	author = {H. Ammari and H. Kang},
	doi = {10.1007/978-0-387-71566-7},
	publisher = {Springer},
	series = {Applied Mathematical Sciences},
	title = {Polarization and Moment Tensors: With Applications to Inverse Problems and Effective Medium Theory},
	volume = {162},
	year = {2007}}

@article{Cheney2001,
	author = {M. Cheney},
	journal = {Inverse Problems},
	pages = {591--595},
	title = {The linear sampling method and the MUSIC algorithm},
	volume = {17},
	year = {2001}}

@techreport{Devaney2000,
	author = {A. J. Devaney},
	institution = {Northeastern University},
	title = {Super-resolution processing of multi-static data using time reversal and {MUSIC}},
	type = {Preprint},
	year = {2000}}

@article{FazliNakhkash2012,
	author = {R. Fazli and M. Nakhkash},
	doi = {10.1088/0266-5611/28/7/075012},
	journal = {Inverse Problems},
	number = {7},
	pages = {075012},
	title = {An analytical approach to estimate the number of small scatterers in 2D inverse scattering problems},
	volume = {28},
	year = {2012}}

@article{QuinlanBarbotLarzabal2006,
	author = {A. Quinlan and J.-P. Barbot and P. Larzabal},
	doi = {10.1121/1.2180207},
	journal = {The Journal of the Acoustical Society of America},
	number = {4},
	pages = {2220--2225},
	title = {Automatic determination of the number of targets present when using the time reversal operator},
	volume = {119},
	year = {2006}}

@article{GriesmaierHankeSylvester2014,
	author = {R. Griesmaier and M. Hanke and J. Sylvester},
	doi = {10.1137/120891381},
	journal = {SIAM Journal on Numerical Analysis},
	number = {1},
	pages = {343--362},
	title = {Far Field Splitting for the Helmholtz Equation},
	volume = {52},
	year = {2014}}

@article{Hanke2012,
	author = {M. Hanke},
	doi = {10.1137/110823985},
	journal = {SIAM Journal on Imaging Sciences},
	number = {1},
	pages = {465--482},
	title = {One Shot Inverse Scattering via Rational Approximation},
	volume = {5},
	year = {2012}}

@article{GriesmaierHanke2025,
	author = {R. Griesmaier and M. Hanke},
	doi = {10.1137/24M1709753},
	journal = {SIAM Journal on Imaging Sciences},
	number = {2},
	pages = {881--905},
	title = {One Shot Inverse Scattering Revisited},
	volume = {18},
	year = {2025}}

@article{Bao2015,
	author = {G. Bao and P. Li and J. Lin and F. Triki},
	doi = {10.1088/0266-5611/31/9/093001},
	journal = {Inverse Problems},
	number = {9},
	pages = {093001},
	title = {Inverse scattering problems with multi-frequencies},
	volume = {31},
	year = {2015}}

@article{Deng2026_a,
	author = {Z. Deng and A. Qian and X. Yang},
	journal = {arXiv preprint arXiv:2606.15065},
	title = {A {Hankel} determinant zero-order principle for source counting in an inverse heat point-source problem},
	year = {2026}}

@article{Deng2026_b,
	author = {Z. Deng and X. Yang and A. Qian},
	journal = {arXiv preprint arXiv:2606.21815},
	title = {A moment-{Hankel} rank method for identifying the number of point sources in the heat equation},
	year = {2026}}

@article{Hua1990,
	author = {Y. Hua and T. K. Sarkar},
	journal = {IEEE Transactions on Acoustics, Speech, and Signal Processing},
	number = {5},
	pages = {814--824},
	title = {Matrix pencil method for estimating parameters of exponentially damped/undamped sinusoids in noise},
	volume = {38},
	year = {1990}}

@article{Potts2010,
	author = {D. Potts and M. Tasche},
	journal = {Signal Processing},
	number = {5},
	pages = {1631--1642},
	title = {Parameter estimation for exponential sums by approximate {Prony} method},
	volume = {90},
	year = {2010}}

@article{Potts2013,
	author = {D. Potts and M. Tasche},
	journal = {Electronic Transactions on Numerical Analysis},
	pages = {204--224},
	title = {Parameter estimation for multivariate exponential sums},
	volume = {40},
	year = {2013}}

@article{Peter2013,
	author = {T. Peter and G. Plonka},
	journal = {Inverse Problems},
	number = {2},
	pages = {025001},
	title = {A generalized {Prony} method for reconstruction of sparse sums of eigenfunctions of linear operators},
	volume = {29},
	year = {2013}}

@article{Kunis2019,
	author = {S. Kunis and H. M. M\"oller and U. von der Ohe},
	journal = {SMAI Journal of Computational Mathematics},
	pages = {87--97},
	title = {{Prony}'s method on the sphere},
	volume = {S5},
	year = {2019}}

@article{CuytLee2024,
	author = {A. Cuyt and W.-s. Lee},
	doi = {10.1007/s11075-023-01564-3},
	journal = {Numerical Algorithms},
	number = {1},
	pages = {31--72},
	title = {Multiscale matrix pencils for separable reconstruction problems},
	volume = {95},
	year = {2024}}

\end{document}